\documentclass[12pt]{amsart}
\usepackage{amssymb,amsmath,latexsym}

\usepackage{amsfonts}
\usepackage{amscd}

\usepackage{graphicx,array}
\usepackage[all]{xy}
\numberwithin{equation}{section}

\def\ga{\mathfrak{a}}

\def\gf{\mathfrak{f}}
\def\mfg{\mathfrak{g}}
\def\gh{\mathfrak{h}}

\def\gk{\mathfrak{k}}
\def\gl{\mathfrak{l}}

\def\gp{\mathfrak{p}}

\def\gs{\mathfrak{s}}
\def\gS{\mathfrak{S}}

\def\Ad{{\rm Ad}}

\def\trace{\mathrm{Tr}\,}

\newcommand{\res}{\mathrm{res}}

\newcommand{\Wn}{W (\fg_n, \fh_{n})}

\renewcommand{\Im}{\mathop{\rm Im} }

\def\C{\mathbb{C}}

\def\H{\mathbb{H}}

\def\N{\mathbb{N}}

\def\R{\mathbb{R}}

\def\Z{\mathbb{Z}}

\newcommand{\id}{\mathrm{id}}
\newcommand{\gL}{\Lambda }

\newcommand{\SO}{\mathrm{SO}}
\newcommand{\Sp}{\mathrm{Sp}}
\newcommand{\Spin}{\mathrm{Spin}}
\newcommand{\SU}{\mathrm{SU}}
\newcommand{\U}{\mathrm{U}}
\newcommand{\diag}{\mathrm{diag}}
\newcommand{\Tr}{\mathrm{Tr}\, }
\newcommand{\so}{\mathfrak{so}}
\newcommand{\lsp}{\mathfrak{sp}}
\newcommand{\lsl}{\mathfrak{sl}}
\def\cA{\mathcal{A}}

\def\cF{\mathcal{F}}

\def\cH{\mathcal{H}}
\def\cI{\mathcal{I}}

\def\cS{\mathcal{S}}

\newcommand{\Br}{B_r}
\newcommand{\Brc}{\overline{B_r(0)}}

\newcommand{\Exp}{\mathrm{Exp}}
\newcommand{\PW}{\mathrm{PW}}

\renewcommand{\Re}{\mathrm{Re}}

\newtheorem{theorem}[equation]{Theorem}

\newtheorem{lemma}[equation]{Lemma}

\newtheorem{corollary}[equation]{Corollary}

\newtheorem{proposition}[equation]{Proposition}

\newtheorem{remark}[equation]{Remark}

\newcommand{\fa}{\mathfrak{a}}

\newcommand{\fg}{\mathfrak{g}}
\newcommand{\fh}{\mathfrak{h}}
\newcommand{\fk}{\mathfrak{k}}
\newcommand{\fl}{\mathfrak{l}}
\newcommand{\fm}{\mathfrak{m}}
\newcommand{\fn}{\mathfrak{n}}

\newcommand{\fs}{\mathfrak{s}}

\newcommand{\fz}{\mathfrak{z}}

\newcommand{\rS}{\mathrm{S}}
\newcommand{\rI}{\mathrm{I}}

\def\sideremark#1{\ifvmode\leavevmode\fi\vadjust{\vbox to0pt{\vss
 \hbox to 0pt{\hskip\hsize\hskip1em
\vbox{\hsize2cm\tiny\raggedright\pretolerance10000
 \noindent #1\hfill}\hss}\vbox to8pt{\vfil}\vss}}}%

\begin{document}

\title[Weyl Group Invariants]{Weyl Group Invariants and Application to
Spherical Harmonic Analysis on Symmetric Spaces}

\author{Gestur \'{O}lafsson}
\address{Department of Mathematics, Louisiana State University, Baton Rouge,
LA 70803, U.S.A.}
\email{olafsson@math.lsu.edu}
\thanks{The research of G. \'Olafsson was supported by NSF grants
DMS-0402068 and DMS-0801010}

\author{Joseph A. Wolf}
\address{Department of Mathematics, University of California, Berkeley,
CA 94707, U.S.A.}
\email{jawolf@math.berkeley.edu}
\thanks{The research of J. A. Wolf was partially supported by NSF grant
DMS-0652840}


\date{}

\begin{abstract}
Polynomial invariants are fundamental objects in analysis on
Lie groups and symmetric spaces. Invariant differential
operators on symmetric spaces are described by Weyl group
invariant polynomial.
In this article we give a simple criterion that
ensure that the restriction of invariant polynomials
to subspaces is surjective.  We apply our criterion to
problems in Fourier analysis on projective/injective limits,
specifically to theorems of Paley--Wiener type.
\end{abstract}
\maketitle

\section*{Introduction} \label{sec0}
\setcounter{equation}{0}

\noindent
Invariant polynomials play a fundamental role in several branches of
mathematics. A well known example related to the topic
of this article comes from the representation theory of
semisimple Lie groups and from the related analysis on Riemannian symmetric
spaces. Let $G$ be a connected semisimple real Lie group
with Lie algebra $\fg$. Let $\fh\subset \fg$ be a Cartan subalgebra.
Then the algebra of $G$--invariant polynomials on $\fg$ is isomorphic
to the center of the universal enveloping algebra of $\fg$. Also, the
restriction of invariant polynomials to $\fh$ is an isomorphism onto the
algebra of Weyl group invariant polynomials on $\fh$. Replace $G$ by
a Riemannian symmetric space $M=G/K$ corresponding to a
Cartan involution $\theta$ and replace $\fh$  by
a maximal abelian subspace $\fa$ in $\fs:=\{X\in\fg\mid \theta (X)=-X\}$. Then
the Weyl group invariant polynomials correspond to the invariant
differential operators on $M$. They are therefor closely related to
harmonic analysis on $M$, in particular to the determination of
the spherical functions on $M$.

In general we need $\fa \subset \fh$ and $\theta\fh = \fh$.  For this, of
course, we need only choose $\fh$ to be a Cartan subalgebra of the centralizer
of $\fa$.

Denote by $W(\fg,\fh)$ the Weyl group of $\fg$ relative to $\fh$,
$W(\fg,\fa)$ the ``baby'' Weyl group of $\fg$ relative to $\fa$,
$W_\fa (\fg,\fh)=\{w\in W(\fg,\fh)\mid w(\fa )= \fa\}$, $\rI(\fg )$
the algebra of $W(\fg,\fh)$--invariant polynomials on $\fh$ and
finally $\rI(\fa )$ the algebra of $W(\fg,\fa)$--invariant polynomials on
$\fa$.   It is well known for all semisimple
Lie algebras that $W_\fa (\fg,\fh)|_{\fa}=W (\fg,\fa)$.  In \cite{He1964}
Helgason showed for all classical semisimple Lie algebras
that $\rI(\fh)|_\fa = \rI(\fa)$. As an application, this shows that in most
cases the invariant differential operators on $M$ come from elements
in the center of the universal enveloping algebra of $\fg$.

In this article we discuss similar restriction problems
for the case of pairs of Lie groups $G_n\subset G_k$ and
symmetric spaces $M_n\subset M_k$. We use the above notation with
indices ${}_n$ respectively ${}_k$. The first question is
about restriction from $\fh_k$ to $\fh_n$. It is clear that neither
does the group $W_{\fh_n}(\fg_k,\fh_k)$ restrict to $W (\fg_n,\fh_n)$ in
general, nor is $\rI(\fh_k)|_{\fh_n}=\rI (\fh_n)$. To make this work, we
introduce the notion that $\fg_k$ is a prolongation of $\fg_n$ using
the Dynkin diagram of simple Lie classical Lie algebras.
In terms of restricted
roots, that means that either the rank and restricted root system
of the large and the small symmetric spaces are the same, or
roots are added to the left end of the Dynkin diagram.  The result is
that both symmetric spaces have the same type of root system but the
larger one can have higher rank. In that case the restriction
result above holds for all cases except when the restricted root systems
are of type $D$. This includes
all the cases of classical Lie groups of the same type. If $G_k$ is
a prolongation of $G_n$, then $W_{\fh_n}(\fg_k,\fh_k)|_{\fh_n}=W (\fg_n,\fh_n)$ and
$\rI(\fh_n)|_{\fh_n}=\rI (\fh_n)$, except in the case of simple algebras
of type $D$, where a parity condition is needed, i.e., we have to extend
the Weyl group by incorporating odd sign changes for simple factors of type $D$.
The resulting finite group is denoted by $\widetilde{W} (\fg,\fh)$. Then,
in all classical cases, the
$\widetilde{W}(\fg_k,\fh_k)$-invariant polynomials restrict to
$ \widetilde{W} (\fg_n,\fh_n)$-invariant polynomials. We also
show that $\widetilde{W}_\fa(\fg,\fh)|_\fa = \widetilde{W}(\fg,\fa)$.

When a compact symmetric space $M_k$ is a prolongation of another, say
$M_n$, we prove surjectivity for restriction of
smooth functions supported in a ball of a given radius $r$ on $M_k$ to
smooth functions supported in a ball of radius $r$ on $M_n$, resulting is
a corresponding restriction result on their Fourier transform spaces.
Using results on conjugate and cut locus of compact symmetric spaces
we show that the radius of injectivity in the symmetric spaces forming
a direct system, related by prolongation, is constant.  If $R$ is that
radius then the condition on the support size $r$ is given in terms of
$R$, thus constant for the direct system, and this allows us to
carry the finite dimensional Paley--Wiener theorem to the limit.

The classical Paley--Wiener Theorem describes the growth of the
Fourier transform of a function $f \in C_c^\infty(\R^n)$ in terms of the
size of its support.  Helgason and Gangolli generalized it to Riemannian
symmetric spaces of noncompact type, Arthur extended it to semisimple Lie
groups, van den Ban and Schlichtkrull made the extension to pseudo-Riemannian
reductive symmetric spaces, and finally \'Olafsson and Schlichtkrull
worked out the corresponding result for compact Riemannian symmetric spaces.
Here we extend their result to a class of infinite dimensional
Riemannian symmetric spaces, the classical direct limits compact symmetric
spaces.  The main idea is to combine the results of \'Olafsson and Schlichtkrull
with Wolf's results on direct limits
$\varinjlim M_n$ of Riemannian symmetric spaces and limits of the
corresponding function spaces on the $M_n$.

Of course compact support in the Paley--Wiener Theorem is
irrelevant for functions on a compact
symmetric space, and there one concentrates on the radius of the support.
The Fourier transform space is interpreted as the parameter space for
spherical functions, the linear dual space of the complex span of the
restricted roots.  When we pass to direct limits it is crucial that
these ingredients be properly normalized.  In order to do this
we introduce the notion of propagation for pairs of root systems, pairs
of groups, and pairs of symmetric spaces.

In Section \ref{sec1} we recall some basic facts concerning Paley--Wiener
theorems, their behavior under finite symmetry groups, and restrictions
of Paley--Wiener spaces.  In order to apply this to direct systems of
symmetric spaces, in Section \ref{sec2} we introduce the notion of
propagation and examine the corresponding invariants explicitly for each
type of root system.  The main result, Theorem \ref{th-AdmExt},
summarizes the facts on restriction of Weyl groups for
propagation of symmetric spaces. The proof is by case by case consideration
of each simple root system.

In Section \ref{sec3} we prove surjectivity of Weyl group invariant
polynomials for propagation of symmetric spaces.  As mentioned above,
this is analogous to Helgason's result on restriction of invariants from
the full Cartan $\fh$ of $\fg$ to the Cartan $\fa$ of $(\fg,\fk)$.

In Section \ref{sec4} we apply our results on Weyl group invariants to
Fourier analysis on Riemannian symmetric spaces of noncompact type.  The
main result is Theorem \ref{th-ProjLimNonCompact}, the Paley--Wiener
Theorem for classical direct limits of those spaces.
As indicated earlier, a $\Z_2$ extension of the Weyl group is needed in
case of root systems of type $D$. The extension can be realized by
an automorphism $\sigma$ of the of the Dynkin diagram. We show that
there exists an automorphism $\widetilde{\sigma}$ of $G$ or a double cover
such that $d\widetilde{\sigma}|_\fa =\sigma$ and the spherical function
with spectral parameter $\lambda$ satisfies
$\varphi_\lambda (\widetilde{\sigma}(x))
=\varphi_{\sigma'(\lambda )}(x)$.

In Section \ref{sec5} we set up the basic surjectivity of the direct
limit Paley--Wiener Theorem for the classical sequences $\{SU(n)\}$,
$\{SO(2n)\}$, $\{SO(2n+1)\}$ and $\{Sp(2n)\}$.  The main tool is
Theorem \ref{re-inj-rad}, the calculation of the injectivity radius; it
turns out to be a simple constant ($\sqrt{2}\,\pi$ or $2\pi$) for each
of those series.  The main result is Theorem \ref{th-CknSurjective}, which
sets up the projective systems of functions used in the Paley--Wiener
Theorem for $SU(\infty)$, $SO(\infty)$ and $Sp(\infty)$.  All this is needed
when we go to limits of symmetric spaces.  Theorem \ref{l-inductiveSystemOfRep},
the main result of Section \ref{sec6},
sets up the sequence of function spaces corresponding to a direct system
$\{M_n\}$ of compact Riemannian symmetric spaces in which $M_k$ propagates
$M_n$ for $k \geqq n$. On the way we show that for compact symmetric
spaces the map
$Q: C^\infty (G)^G\to C^\infty (G/K)^K$, $Q(f)(xK)=\int_K f(xk)\, dk$, which
is surjective, is in fact surjective as a map $C^\infty_r(G)^G \to C_r^\infty (G/K)^K$,
where the subscript ${}_r$ denotes the size of the support.

Finally, in Section \ref{sec7}, we relate the spherical Fourier transforms
for the sequence $\{M_n\}$, show how the injectivity radii remain constant
on the sequence, and prove the Paley--Wiener Theorem \ref{t: PW}, and a
Paley--Wiener Theorem \ref{th-PWcompactII}, for direct limits
$M_\infty = \varinjlim M_n$ of compact Riemannian symmetric spaces in which
$M_k$ propagates $M_n$ for $k \geqq n$.  Along the way we obtain a stronger
form, Theorem \ref{stronger}, of one of the key ingredients in the proof
of the surjectivity.

Our discussion of direct limit Paley--Wiener Theorems involves function
space maps that have a somewhat complicated relation \cite{W2009}
to the $L^2$ theory of \cite{W2008a}.  This is
discussed in Section \ref{sec8}, where we compare our maps with the
partial isometries of \cite{W2008a}.

\section{Polynomial Invariants and Restriction of
Paley-Wiener spaces}\label{sec1}\setcounter{equation}{0}

\noindent
Let $E\cong \R^n$ be a finite dimensional Euclidean space
and $E_\C\cong \C^n$ its complexification. Denote by
$C_r^\infty ( E )$ the space of smooth functions on
$E$ with support in a closed ball $\Brc$ of radius $r>0$. Denote
by $\PW_r(E_\C)$ the space of holomorphic function on $E_\C$ with
the property that for each $n\in\Z^+$ there exists
a constant $C_n>0$ such that
\begin{equation}\label{eq-DefPW}
|F(z)|\leqq C_n (1+|z|)^{-n}e^{r|\Im z |}\, .
\end{equation}

Let $\langle x,y\rangle_E = \langle x,y\rangle =x\cdot y$ denote the inner
product on $E$ as well as its $\C$--bilinear extension to $E_\C$.
Denote by $\mathrm{O}(E)$ the orthogonal group of $E$
with respect to this inner product. If $T:E\to E$ is $\R$-linear
then we will also view $T$ as a complex linear map on $E_\C$.
If $w \in \mathrm{O}(E)$ and $f$ is a function on $E$ or $E_\C$
then $L_w(f)$ denotes the (left) translate of $f$ by $w$,
$L_w(f)(x)=f(w^{-1}x)$.  If $G$ is a subgroup of $\mathrm{O}(E)$
and $L_wf=f$ for all $w\in G$ then we say that
$f$ is $G$--invariant.  For a $G$-module $V$ set
\begin{equation}\label{eq-invariants}
V^G=\{v\in V\mid g\cdot v=v\text{ for all } g\in G\}\, .
\end{equation}
In particular $\PW_r(E_\C)^G$ and $C_r^\infty (E)^G$ are well defined.

We normalize
the Fourier transform on $E$ as
\begin{equation}\label{eq-FourierTransform}
\cF_E(f)(\lambda )=\widehat{f}(\lambda )
=(2\pi )^{- n/2}\int_E f(x)e^{-i\lambda \cdot x}\, dx.
\end{equation}
The Paley--Wiener Theorem says that
$\mathcal{F}_E :C_r^\infty (E)^G\to \PW_r(E_\C)^G$
is an isomorphism.

Denote by $\rS (E)$ the algebra of polynomial functions on $E$ and
$\rI (E)=\rI_G (E)=\rS (E)^G$ the algebra of $G$--invariant invariant
polynomial functions on $E$.

{}From now on we assume that $F$ is another
Euclidean space and that $E\subseteqq F$.
We will always assume that the inner products on $E$ and
$F$ are chosen so that $\langle x,y\rangle_E=\langle x,y\rangle_F$
for all $x,y\in E$. Furthermore, if
$W(E)$ and $W(F)$ are closed subgroups of the respective orthogonal groups
acting on $E$ and $F$, then $$W_E(F)=\{w\in W(F)\mid w(E)=E\}$$
is the subgroup of $W(F)$ that maps $E$ into $E$.
We will always assume that
$W(E)$ and $W(F)$ are generated by reflections $s_\alpha : v\mapsto
v-\frac{2(\alpha ,v)}{(\alpha ,\alpha )}\alpha$, for $\alpha$ in a root system
in $E$ respectively $F$. However, it should be pointed out that the Cowling result, see
below, holds for arbitrary closed subgroup of $\mathrm{O} (E)$ respectively $\mathrm{O} (F)$.

We recall the following theorems of Cowling \cite{cowling}
and Rais \cite{Rais} in the form that we need in the
sequel.  Cowling states his result for finite groups but his argument is
valid for compact groups.

\begin{theorem}[Cowling]\label{th-cowling} The restriction map
$\PW_r(F_\C)^{W_E(F)} \to \PW_r(E_\C)^{W_E (F)|_{E_\C}}$, given by
$F\mapsto F|_{E_\C}$, is surjective.
\end{theorem}

\begin{theorem}[Rais]\label{th-rais}
Let $P_1,\ldots ,P_n$ be a basis for $\rS (F)$ over
$I_{W(F)}(F)$. If $F\in\PW_r(F_\C)$
there exist  $\Phi_1,\ldots ,\Phi_n\in\PW_r(F_\C)^{W(F)}$ such
that
$$F=P_1\Phi_1+ \ldots + P_n\Phi_n\, .$$
\end{theorem}

If $W_E(F)|_E=W(E)$ then Cowling's Theorem implies that the restriction map
\[\PW_r(F_\C)^{W_E(F)} \to \PW_r(E_\C)^{W_E (F)|_{E_\C}}\, ,\quad F\mapsto F|_{E_\C}\, ,\]
is surjective. If we, on the right hand side, replace $W_E(F)$ by the full
group $W(F)$ then the subspace of invariant functions becomes smaller, so one would
in general not expect the restriction map to remain surjective. The following
theorem gives a sufficient condition for that to happen.

\begin{theorem}\label{th-IsoPW1} Let the notation be as above. Assume that
\begin{enumerate}
\item $W_E(F)|_E =W(E)$
\item The restriction map $\rI_{W(F)} (F)\to \rI_{W(E)}(E)$ is surjective.
\end{enumerate}
Then the restriction map
$$
\PW_r(F_\C)^{W(F)}\to \PW_r(E_\C)^{W(E)}\, \text{, given by }
F\mapsto F|_{E_\C}\, ,
$$
is surjective.
\end{theorem}

\begin{proof} It follows from assumption (1) that if $F\in\PW_r(F_\C)^{W(F)}$ then
$F|_E\in \PW_r(E_\C)^{W(E)}$.

Now, let $G\in \PW_r(E_\C)^{W(E)}$. By Theorem \ref{th-cowling} and
assumption (1) there
exists a function $\widetilde{G}\in \PW_r(F_\C)^{W_E(F)}$ such that
$\widetilde{G}|_{E_\C}=G$. By Theorem \ref{th-rais}, there exist
$\Phi_1,\ldots ,\Phi_n\in \PW_r(F_\C)^{W(F)}$ and
polynomials $P_1,\ldots ,P_n\in \rS (F)$ such that
$$\widetilde G =P_1\Phi_1+\ldots + P_n\Phi_n\, .$$
But then
$$G=\widetilde{G}|_{E_\C}=(P_1|_{E_\C})(\Phi_1|_{E_\C})+\ldots + (P_n|_{E_\C})(\Phi_n|_{E_\C})\, .$$
As $W (E) =W_E(F)|_E$, $G$ is $W (E)$--invariant and the functions $\Phi_j$ are $W(F)$--invariant,
we can assume that $P_j|_{E_\C} \in \rI_{W (E)}(E)$.
By assumption (2) there exists $Q_j\in \rI_{W(F)} (F)$ such that $Q_j|_{E_\C}=P_j|_{E_\C}$. But then
$$\Phi:=Q_1\Phi_1+\ldots +Q_r\Phi_r\in \PW_r(F_\C)^{W(F)}
\text{ satisfies } \Phi|_{E_\C}=G\, .$$
Hence the restriction map is surjective.
\end{proof}

\begin{remark}{\em
We note that (1) above does not imply (2) in Theorem \ref{th-IsoPW1}.
Let $G/K$ be a
semisimple symmetric space of the noncompact type. Let $\fg =\fk \oplus \fs$
be the corresponding Cartan decomposition. Thus, there exists an involution
$\theta : \fg \to \fg$ such that $\fk$, the Lie algebra of $K$, is
the $+1$-eigenspace of $\theta$ and $\fs = \{X\in \fg\mid \theta (X) = -X\}$.
Let $\fa$ be a maximal abelian subspace of $\fs$ and let
$\Sigma (\fg,\fa)$
be the set of roots of $\fa$ in $\fg$. We write $W(\fg,\fa )$  for
the corresponding
Weyl group. Let $\fh$ be a Cartan subalgebra of $\fg$ containing $\fa$. Then
$\fh = \fh_k\oplus \fa$, where $\fh_k=\fh\cap \fk$. Let $\Delta (\fg_\C,\fh_\C)$
be the roots of $\fh_\C$ in $\fg_\C$ and let $W (\fg,\fh)$ be the corresponding
Weyl group. Then by \cite[p. 366]{He1984},
\[W (\fg,\fa)=\{w|_{\fa}\mid w\in W (\fg,\fh) \text{ such that }  w(\fa )=\fa\}
=W_{\fa}(\fg,\fh)|_\fa.\]
But for some of the exceptional algebras (2) can fail; see \cite{He1992} for exact statement.
} \hfill $\diamondsuit$
\end{remark}

Let $n=\dim E$ and $m=\dim F$. Denote by $\cF_E$ respectively $\cF_F$ the
Euclidean Fourier transforms on $E$ and $F$.
The following map $C$ was denoted by $P$ in \cite{cowling}.

\begin{corollary}[Cowling]\label{co-Cowling} Let the assumptions be as above.
Then the  map
$$
C: C_r^\infty (F)^{W(F)}\to C_r^\infty (E)^{W(E)}\, ,\text{ given by }
C f(x)=\int_{E^\perp} f(x,y)\, dy,
$$
is surjective.
\end{corollary}
\begin{proof} We follow \cite{cowling}. Let $c=(2\pi )^{(n-m)/2}$ Let $g\in C_r^\infty (E)^{W(E)}$ and $G=\cF_E(g)$ its
Fourier transform in $\PW_r(E_\C)^{W(E)}$. Let $F\in \PW_r(F_\C)^{W(F)}$ be
such that $F|_E=c^{-1} G$ and let $f$ be the inverse Fourier transform of
$F$. We claim that $g=Cf$. For that we note
\begin{eqnarray*}
\cF_E(g)(\lambda )&=&c^{-1} F(\lambda ,0)\\
&=&c^{-1} (2\pi )^{-m/2} \int_{E_x}\int_{E^\perp_y} f(x,y)e^{-i\lambda\cdot x}\, dxdy\\
&=&(2\pi)^{-n/2}\int_E C f(x)e^{-i\lambda\cdot x}\, dx\\
&=&\cF_E( C f)(\lambda )\, .
\end{eqnarray*}
The claim follows now, as obviously $Cf$ has compact support and the Fourier transform
is injective on the space of compactly supported functions.
\end{proof}

\begin{theorem}\label{th-projectiveLimitPW}
Let $\{E_j\}$ be a sequence of Euclidean spaces such that
$E_j\subseteqq E_{j+1}$ and such that hypotheses $(1)$ and $(2)$
of {\rm Theorem \ref{th-IsoPW1}} are satisfied for
each pair $(E_j,E_k)$, $k\geqq j$.  Let
$P_{k,j} :\PW_r(E_{k,\C})^{W(E_{k})}\to
\PW_r (E_{j,\C})^{W(E_j)}$ be the restriction map. Then
$\{\PW_r (E_{j,\C})^{W(E_j)}, P_j\}$ is a projective system and
$\varprojlim  \{\PW_r (E_{j,\C})^{W(E_j)}\} \not= \{0\}$.
\end{theorem}
\begin{proof} It is clear that $\{\PW_r (E_{j,\C})^{W(E_j)}, P_{n,j}\}$
is a projective system.
Fix $j$ and let $F\in \PW_r (E_{j,\C})^{W(E_j)}$, $F\not= 0$. Recursively
choose $F_k \in \PW_r(E_{k,\C})^{W(E_k)}$, $k\geqq j$ such that
$F_{k+1}|_{E_{k,\C}}=F_k$.  Then the sequence $\{F_k\}$ is a non-zero element of
$\varprojlim \PW_r (E_{j,\C})^{W(E_j)}$.
\end{proof}

\begin{theorem} Let $\{E_j\}$ be a sequence of Euclidean spaces such that
$E_j\subseteqq E_{j+1}$ and such that hypotheses {\rm (1)} and  {\rm (2)}
of  {\rm Theorem \ref{th-IsoPW1}} are satisfied for
each pair $(E_j,E_k)$, $k\geqq j$. Define $C_{k,j} :C^\infty_r(E_{k})^{W(E_{k})}\to
C^\infty_r (E_j)^{W(E_j)}$ by
\[[C_{k,j}(f)](x)=\int_{E_j^\perp} f(x,y)\, dy\, .\]
Then the maps $C_{k,j}$ are surjective,
$\{C^\infty_r (E_j)^{W(E_j)}, C_{k,j}\}$ is a projective system, and
its limit satisfies $\varprojlim C^\infty_r (E_j)^{W(E_j)}\not= \{0\}$.
\end{theorem}

\begin{proof} The proof is the same as that of Theorem
\ref{th-projectiveLimitPW}, making use of Corollary \ref{co-Cowling}.
\end{proof}

\begin{remark} The last two theorems remain valid if the assumptions holds for
a cofinite sequence $\{E_j\}_{j\in J}$. \hfill $\diamondsuit$
\end{remark}
\section{Restriction of Invariants for Classical Simple Lie Algebras}
\label{sec2}
\noindent
We will now apply this to the classical simple Lie algebras and related
symmetric spaces.  Let $\fg_n$ be a simple Lie algebra of classical type
and let $\fh_n\subset \fg_n$ be a Cartan subalgebra. Let
$\Delta_n = \Delta(\fg_n, \fh_n)$ be the set of roots of $\fh_{n,\C}$ in
$\fg_{n,\C}$ and $\Psi_n = \Psi(\fg_n, \fh_n)$ a set of simple roots. We label
the corresponding
Dynkin diagram so that $\alpha_1$ is the \textit{right} endpoint.
If $\fg_n\subseteqq \fg_k$ then we chose $\fh_n$ and $\fh_k$ so
that $\fh_n=\fg_n\cap \fh_k$.
We say that $\fg_k$ \textit{propagates} $\fg_n$, if
$\Psi_k$ is constructed from $\Psi_n$ by adding simple roots to the \textit{left} end
of the Dynkin diagrams. Thus
\begin{equation}\label{rootorder}
\begin{aligned}
&\begin{tabular}{|c|l|c|}\hline
$\Psi_n=A_n$ &
\setlength{\unitlength}{.5 mm}
\begin{picture}(155,18)
\put(48,2){\circle{2}}
\put(45,5){$\alpha_n$}
\put(49,2){\line(1,0){23}}
\put(73,2){\circle{2}}
\put(70,5){$\alpha_{n-1}$}
\put(74,2){\line(1,0){23}}
\put(98,2){\circle{2}}
\put(95,5){$\alpha_{n-2}$}
\put(99,2){\line(1,0){13}}
\put(117,2){\circle*{1}}
\put(120,2){\circle*{1}}
\put(123,2){\circle*{1}}
\put(129,2){\line(1,0){13}}
\put(143,2){\circle{2}}
\put(140,5){$\alpha_1$}
\end{picture}
&$n\geqq 1$
\\
\hline
$\Psi_k=A_k$&
\setlength{\unitlength}{.5 mm}
\begin{picture}(155,18)
\put(5,2){\circle{2}}
\put(2,5){$\alpha_{k}$}
\put(6,2){\line(1,0){13}}
\put(24,2){\circle*{1}}
\put(27,2){\circle*{1}}
\put(30,2){\circle*{1}}
\put(34,2){\line(1,0){13}}
\put(48,2){\circle{2}}
\put(45,5){$\alpha_n$}
\put(49,2){\line(1,0){23}}
\put(73,2){\circle{2}}
\put(70,5){$\alpha_{n-1}$}
\put(74,2){\line(1,0){23}}
\put(98,2){\circle{2}}
\put(95,5){$\alpha_{n-2}$}
\put(99,2){\line(1,0){13}}
\put(117,2){\circle*{1}}
\put(120,2){\circle*{1}}
\put(123,2){\circle*{1}}
\put(129,2){\line(1,0){13}}
\put(143,2){\circle{2}}
\put(140,5){$\alpha_1$}
\end{picture}
&$k\geqq n$
\\
\hline
\end{tabular}\\
&\begin{tabular}{|c|l|c|}\hline
$\Psi_n=B_n$ &
\setlength{\unitlength}{.5 mm}
\begin{picture}(120,18)
\put(48,2){\circle{2}}
\put(45,5){$\alpha_n$}
\put(49,2){\line(1,0){23}}
\put(73,2){\circle{2}}
\put(70,5){$\alpha_{n-1}$}
\put(74,2){\line(1,0){13}}
\put(93,2){\circle*{1}}
\put(96,2){\circle*{1}}
\put(99,2){\circle*{1}}
\put(104,2){\line(1,0){13}}
\put(118,2){\circle{2}}
\put(115,5){$\alpha_2$}
\put(119,2.5){\line(1,0){23}}
\put(119,1.5){\line(1,0){23}}
\put(143,2){\circle*{2}}
\put(140,5){$\alpha_1$}
\end{picture}&$n\geqq 2$
\\
\hline
$\Psi_k=B_k$&
\setlength{\unitlength}{.5 mm}
\begin{picture}(155,18)
\put(5,2){\circle{2}}
\put(2,5){$\alpha_{k}$}
\put(6,2){\line(1,0){13}}
\put(24,2){\circle*{1}}
\put(27,2){\circle*{1}}
\put(30,2){\circle*{1}}
\put(34,2){\line(1,0){13}}
\put(48,2){\circle{2}}
\put(45,5){$\alpha_n$}
\put(49,2){\line(1,0){23}}
\put(73,2){\circle{2}}
\put(70,5){$\alpha_{n-1}$}
\put(74,2){\line(1,0){13}}
\put(93,2){\circle*{1}}
\put(96,2){\circle*{1}}
\put(99,2){\circle*{1}}
\put(104,2){\line(1,0){13}}
\put(118,2){\circle{2}}
\put(115,5){$\alpha_2$}
\put(119,2.5){\line(1,0){23}}
\put(119,1.5){\line(1,0){23}}
\put(143,2){\circle*{2}}
\put(140,5){$\alpha_1$}
\end{picture}
&$k\geqq n$\\
\hline
\end{tabular} \\
&\begin{tabular}{|c|l|c|}\hline
$\Psi_n=C_n$ &
\setlength{\unitlength}{.5 mm}
\begin{picture}(155,18)
\put(48,2){\circle*{2}}
\put(45,5){$\alpha_n$}
\put(49,2){\line(1,0){23}}
\put(73,2){\circle*{2}}
\put(70,5){$\alpha_{n-1}$}
\put(74,2){\line(1,0){13}}
\put(93,2){\circle*{1}}
\put(96,2){\circle*{1}}
\put(99,2){\circle*{1}}
\put(104,2){\line(1,0){13}}
\put(118,2){\circle*{2}}
\put(115,5){$\alpha_2$}
\put(119,2.5){\line(1,0){23}}
\put(119,1.5){\line(1,0){23}}
\put(143,2){\circle{2}}
\put(140,5){$\alpha_1$}
\end{picture}
& $n \geqq 3$
\\
\hline
$\Psi_k=C_k$ &
\setlength{\unitlength}{.5 mm}
\begin{picture}(155,18)
\put(5,2){\circle*{2}}
\put(2,5){$\alpha_{k}$}
\put(6,2){\line(1,0){13}}
\put(24,2){\circle*{1}}
\put(27,2){\circle*{1}}
\put(30,2){\circle*{1}}
\put(34,2){\line(1,0){13}}
\put(48,2){\circle*{2}}
\put(45,5){$\alpha_n$}
\put(49,2){\line(1,0){23}}
\put(73,2){\circle*{2}}
\put(70,5){$\alpha_{n-1}$}
\put(74,2){\line(1,0){13}}
\put(93,2){\circle*{1}}
\put(96,2){\circle*{1}}
\put(99,2){\circle*{1}}
\put(104,2){\line(1,0){13}}
\put(118,2){\circle*{2}}
\put(115,5){$\alpha_2$}
\put(119,2.5){\line(1,0){23}}
\put(119,1.5){\line(1,0){23}}
\put(143,2){\circle{2}}
\put(140,5){$\alpha_1$}
\end{picture}
& $k\geqq n$
\\
\hline
\end{tabular}\\
&\begin{tabular}{|c|l|c|}\hline
$\Psi_n=D_n$ &
\setlength{\unitlength}{.5 mm}
\begin{picture}(155,20)
\put(48,9){\circle{2}}
\put(45,12){$\alpha_n$}
\put(49,9){\line(1,0){23}}
\put(73,9){\circle{2}}
\put(70,12){$\alpha_{n-1}$}
\put(74,9){\line(1,0){13}}
\put(93,9){\circle*{1}}
\put(96,9){\circle*{1}}
\put(99,9){\circle*{1}}
\put(104,9){\line(1,0){13}}
\put(118,9){\circle{2}}
\put(113,12){$\alpha_3$}
\put(119,8.5){\line(2,-1){13}}
\put(133,2){\circle{2}}
\put(136,0){$\alpha_1$}
\put(119,9.5){\line(2,1){13}}
\put(133,16){\circle{2}}
\put(136,14){$\alpha_2$}
\end{picture}
& $n\geqq 4$
\\
\hline
$\Psi_k=D_k$ &
\setlength{\unitlength}{.5 mm}
\begin{picture}(155,20)
\put(5,9){\circle{2}}
\put(2,12){$\alpha_{k}$}
\put(6,9){\line(1,0){13}}
\put(24,9){\circle*{1}}
\put(27,9){\circle*{1}}
\put(30,9){\circle*{1}}
\put(34,9){\line(1,0){13}}
\put(48,9){\circle{2}}
\put(45,12){$\alpha_n$}
\put(49,9){\line(1,0){23}}
\put(73,9){\circle{2}}
\put(70,12){$\alpha_{n-1}$}
\put(74,9){\line(1,0){13}}
\put(93,9){\circle*{1}}
\put(96,9){\circle*{1}}
\put(99,9){\circle*{1}}
\put(104,9){\line(1,0){13}}
\put(118,9){\circle{2}}
\put(113,12){$\alpha_3$}
\put(119,8.5){\line(2,-1){13}}
\put(133,2){\circle{2}}
\put(136,0){$\alpha_1$}
\put(119,9.5){\line(2,1){13}}
\put(133,16){\circle{2}}
\put(136,14){$\alpha_2$}
\end{picture}
& $k\geqq n$
\\
\hline
\end{tabular}
\end{aligned}
\end{equation}

Let $\fg$ and ${}'\fg\subset \fg$ be semisimple Lie algebras. Then
$\fg$ \textit{propagates} ${}'\fg$
if we can number the simple ideals $\fg_j$, $j=1,2, \ldots ,r$, in $\fg$ and
the simple ideals ${}'\fg_i$, $i = 1, 2, \dots , s$, in ${}'\fg$, so that
$\fg_j$ propagates ${}'\fg_j$ for $j=1,\ldots ,s$.

When $\fg_k$ propagates $\fg_n$ as above, they have Cartan subalgebra
$\fh_k$ and $\fh_n$ such that $\fh_n\subseteqq \fh_k$, and we have
choices of root order such that
\[\text{if } \alpha \in \Psi_n \text{ then there is a unique }
\alpha' \in \Psi_k \text{ such that } \alpha'|_{\fh_n} = \alpha.\]
It follows that
\[\Delta_n \subseteqq \{\alpha|_{\fh_n}\mid \alpha\in \Delta_k \text{ and }
\alpha|_{\fh_n}\not= 0\}\, .\]

For a Cartan subalgebra $\fh_\C$ in a simple complex Lie algebra $\fg_\C$
denote by $\fh_\R$ the Euclidean vector space
\[
\fh_\R=\{X\in\fh_\C\mid
\alpha (X) \in \R \text{ for all } \alpha \in \Delta (\fg_\C,\fh_\C)\}\, .
\]

We now discuss case by case the classical simple Lie algebras
and how the Weyl group and the invariants behave under
propagation. The result will be collected in
Theorem \ref{th-AdmExt} below. The corresponding result for
Riemannian symmetric spaces is Theorem \ref{th-AdmExtG/K}.

For $s\in \N$ identify $\R^s$ with its dual. Let $f_1= (0, 0, \ldots ,0,1) $,
\ldots , $f_s=(1,0,0,\ldots ,0)$ be the standard
basis for $\R^s$ numbered in order opposite to the usual one.
We write
\[x=x_1f_1+\ldots +x_s f_s=(x_s,\ldots ,x_1)\]
to indicate that in the following we will be adding zeros to the left
to adjust for our numbering in the Dynkin diagrams. We use the discussion in
\cite[p. 293]{V1974} as a reference for the realization of
the classical Lie algebras.

For a classical simple Lie algebra $\fg$ of rank $n$ denote by $\pi_n$ the
defining representation and
\[F_n(t,X):=\det (t+\pi_n (X))\, .\]
We denote by the same letter the restriction of $F_n(t, \cdot )$ to $\fh_n$.
In this section only we use the following simplified notation:
$W_k = W(\fg_k, \fh_k)$ denotes the usual Weyl group of the pair $(\fg_k,\fh_k)$ and
\[W_{k,n}=
W_{\fh_{n,\R}}(\fg, \fh_{k})=\{w\in W_k\mid w(\fh_{n,\R})=\fh_{n,\R}\}\]
is the subgroup with well defined restriction to $\fh_n$.

\noindent{\textbf{The case
$\mathbf{A_n}$, where $\mathbf{\fg =\mathfrak{sl}(n+1,\C)}$.}} In this case
\begin{equation}\label{an}
\fh_{k,\R}=\{(x_{k+1},\ldots ,x_{1})\in \R^{k+1}\mid x_1+\ldots +x_{k+1}=0\}\, ,
\end{equation}
where $x\in \R^{k+1}$ corresponds to the diagonal matrix
\[x\leftrightarrow \mathrm{diag}(x):=\begin{pmatrix} x_{k+1} &0 & \ldots &0 \cr 0 &x_{k} & & \cr
& &  \ddots & \cr &&& x_1\end{pmatrix}\]
Then $\Delta =\{ f_i-f_j\mid 1\leqq i\not= j\leqq k+1\}$ where $f_\ell$
maps a diagonal matrix to its $\ell^{th}$ diagonal element.  Here
$W (\fg_k,\fh_k)$ is the symmetric group $\gS_{k+1}$, all permutations of
$\{1, \dots , k+1\}$, acting on the $\fh_{k}$ by
\[\sigma\cdot (x_{k+1},\ldots ,x_{1})=
(x_{\sigma^{-1}(k+1)},\ldots ,x_{\sigma^{-1}(1)})\, .\]
We will use the simple root system
\[\Psi (\fg_k,\fh_k) =\{f_j-f_{j-1}\mid j=2,\ldots ,k+1\}\, .\]
The analogous notation will be used for $A_n$. In particular, denoting
the zero vector of length $j$ by $0_j$, we have
\begin{equation}\label{eq-hnInhk}
\fh_{n,\R}=\left \{ (0_{k-n},x_{n+1},\ldots ,x_1)\mid  x_j\in\R\quad
\text{and}\quad \sum_{j=1}^{n+1}x_j=0 \right \}\subset \fh_{k,\R}\, .
\end{equation}
This corresponds to the embedding
\[\mathfrak{sl} (n,\C)\hookrightarrow \mathfrak{sl} (k,\C)\, ,\quad X\mapsto
\begin{pmatrix} 0_{k-n,k-n} & 0\cr
0 & X\end{pmatrix}\, .\]
It follows that
\[W_{k,n}=\gS_{k-n}\times \gS_{n+1}\, .\]
Hence $W_{k,n}|_{\fh_{n,\R}}=W (\fg_n,\fh_)$
and the kernel of the restriction map is the
first factor $\gS_{k-n}$.

According to \cite[Exercise 58, p. 410]{V1974} we have
\[F_k(t,X)=\prod_{j=1}^{k+1}(t+x_j)=
t^{k+1}+\sum_{ \nu=1}^{k+1} p_{k, \nu} (X)t^{\nu -1}\, .\]
The polynomials $p_{k,\nu}$ generate $\rI_{W (\fg_k,\fh_k)}(\fh_{k,\R})$.
By (\ref{eq-hnInhk}), if $X=(0_{k-n},x)\in \fh_{n,\R}$, then
\begin{eqnarray*}
F_k(t,(0_{k-n},x))&=&t^{k+1}+\sum_{\nu=1}^{k+1} p_{k, \nu} (X)t^{\nu -1}\\
&=&t^{k-n}\det (t+\pi_{n}(x))\\
&=&t^{k-n}(t^{n+1} +\sum_{\nu =1}^{n+1} p_{n,\nu }(x)t^{\nu -1})\\
&=&t^{k+1}+\sum_{\nu =k-n+1}^{k+1} p_{n,\nu +n-k}(x)t^{\nu -1}\, .
\end{eqnarray*}
Hence
\[p_{k,\nu}|_{\fh_{n,\R}}= p_{n,\nu +n-k} \text{ for }
k-n+1 \leqq \nu \leqq k\]
and
\[p_{k,\nu}|_{\fh_{n,\R}}=0 \text{ for } 1 \leqq \nu \leqq k-n\, .\]
In particular the restriction map $\rI_{W (\fg_k,\fh_k)}(\fh_{k,\R})\to \rI_{W (\fg_n,\fh_n)}(\fh_{n,\R})$
is surjective.

\noindent
\textbf{The case $\mathbf{B_n}$, where $\mathbf{\fg=\so (2n+1,\C)}$.} In this case
$\fh_{k,\R}=\R^k$ where $\R^k$ is embedded into $\so (2n+1,\C)$ by
\begin{equation}\label{bn}
x\mapsto \begin{pmatrix} 0 & 0 &
   0\cr 0 &\mathrm{diag}(x)& 0\cr 0 & 0 &-\mathrm{diag}(x)\end{pmatrix}\, .
\end{equation}
Here $\Delta_k=\{\pm (f_i\pm f_j)  \mid 1\leqq j < i \leqq k\}\bigcup
\{\pm f_1,\ldots ,\pm  f_k\}$ and we have the positive system
$\Delta_k^+=\{f_i\pm f_j\mid 1\leqq j< i\leqq n\}\cup \{f_1,\ldots ,f_n\}$.
The simple root system is
$\Psi = \Psi(\fg_k,\fh_k) = \{\alpha_1, \dots , \alpha_k\}$ where
\[\text{the simple root } \alpha_1=f_1 \text{, and } \alpha_j=f_{j}-f_{j-1}
\text{ for } 2 \leqq j \leqq k.\]
In this case the Weyl group $W (\fg_k,\fh_k)$ is the semidirect product
$\gS_k\rtimes \{1,-1\}^k$, where
$\gS_k$ acts as before and
\[\{1,-1\}^k\cong (\Z /2\Z )^k=\{\mathbf{\epsilon}=(\epsilon_k,\ldots ,\epsilon_1)\mid \epsilon_j=\pm 1\}\]
acts by sign changes
\[\mathbf{\epsilon}\cdot x=(\epsilon_k x_k,\ldots ,\epsilon_n x_1)\, .\]
Similar notation holds for $\fh_{n,\R}$. Our embedding of
$\fh_{n,\R}\hookrightarrow \fh_{k,\R}$ corresponds to the (non-standard) embedding of
$\so (2n+1,\C)$ into $\so (2k+1,\C)$ given by
\[\begin{pmatrix} 0 & a & b \cr -b^t & A & B\cr -a^t & C & -A^t\end{pmatrix}
\mapsto \begin{pmatrix}0 & 0_{k-n} & a & 0_{k-n} & b  \cr 0_{k-n}^t & 0  & 0 & 0 &
0 \cr
-b^t & 0  & A & 0 & B\cr
0_{k-n}^t & 0  & 0 & 0 &
0\cr
-a^t & 0 & C & 0 & -A^t
\end{pmatrix}
\]
where the zeros stands for the zero matrix of the obvious size and we use
the realization from \cite[p. 303]{V1974}.

We see that
\[W_{k,n}=(\gS_{k-n}\rtimes \{1,-1\}^{k-n})\times
(\gS_n\rtimes \{1,-1\}^n)\, .\]
Thus $W_{k,n}|_{\fh_{n,\R}}= W (\fg_n,\fh_n)$ and the kernel
of the restriction map is
$\gS_{k-n}\rtimes \{1,-1\}^{k-n}$.

For the invariant polynomials we have, again using
\cite[Exercise 58, p. 410]{V1974}, that
\[F_k(t,X)= \det(t + \pi_k(X)) =
t^{2k+1}+\sum_{\nu =1}^k p_{k,\nu} (X)t^{2\nu -1}\]
and the polynomials $ p_{k,\nu}$ freely generate $\rI_{W(\fg_k,\fh_k)}(\fh_{k,\R})$.
For $X \in \fh_k$,  $F_k(t,X)$ is given by
$t\prod_{j=1}^n(t+x_j)(t-x_j)=t\prod_{j=1}^n(t^2-x_j^2)$.
By the same argument as above we have
for $X=(0_{k-n},x)\in \fh_{n,\R}\subseteqq \fh_{k,\R}$:
\begin{eqnarray*}
F_k(t,(0_{k-n},x))&=&t^{2k+1}+\sum_{\nu=1}^k p_{k, \nu} (X)t^{2\nu -1}\\
&=&t^{2(k-n)}\det (t+\pi_{n}(x))\\
&=&t^{2(k-n)}(t^{2n+1} +\sum_{\nu =1}^n p_{n,\nu }(x)t^{2\nu -1})\\
&=&t^{2k+1}+\sum_{\nu =k-n+1}^k p_{n,\nu +n-k}(x)t^{2\nu -1}\, .
\end{eqnarray*}
Hence
\[p_{k,\nu}|_{\fh_{n,\R}}= p_{n,\nu +n-k} \text{ for }
k-n+1 \leqq \nu \leqq k \]
and
\[p_{k,\nu}|_{\fh_{n,\R}}=0 \text{ for } 1 \leqq \nu \leqq k-n\, .\]
In particular, the restriction map $\rI_{W (\fg_k,\fh_k)}(\fh_{k,\R})\to \rI_{W (\fg_n,\fh_n)}(\fh_{n,\R})$
is surjective.

\noindent
\textbf{The case $\mathbf{C_n}$, where $\mathbf{\fg=\lsp (n,\C)}$.} Again
$\fh_{k,\R}=\R^k$ embedded in $\lsp (n,\C)$ by
\begin{equation}\label{cn}
x\mapsto \begin{pmatrix} \diag (x) & 0 \cr 0 &-\diag (x)\end{pmatrix}\, .
\end{equation}
In this case
\[\Delta_k=\{\pm (f_i\pm f_j)  \mid 1\leqq j < i \leqq k\}\bigcup \{\pm 2f_1,\ldots ,\pm 2f_k\}\, .\]
Take $\Delta_k^+=\{f_i-f_j\mid 1\leqq j< i\leqq n\}\cup \{2f_1,\ldots ,2 f_n\}$
as a positive system. Then the simple root system
$\Psi = \Psi (\fg_k,\fh_k) = \{\alpha_1 , \dots , \alpha_k\}$ is given by
\[\text{the simple root } \alpha_1=2f_1 \text{, and } \alpha_j=f_{j}-f_{j-1}
\text{ for } 2 \leqq j \leqq k.\]
The Weyl group $W(\fg_k,\fh_k)$ is again $\gS_k\rtimes \{1,-1\}^k$ and
\[W_{k,n}=(\gS_{k-n}\rtimes \{1,-1\}^{k-n})\times (\gS_n\rtimes\{1,-1\}^n)\, .\]
Thus, $W_{k,n}|_{\fh_{n,\R}}= W (\fg_n,\fh_n)$ and the kernel of the restriction map is
$\gS_{k-n}\rtimes \{1,-1\}^{k-n}$.

For the invariant polynomials we have, again using
\cite[Exercise 58, p. 410]{V1974}, that
\[F_k(t,X)=
t^{2k}+\sum_{\nu =1}^k p_{k,\nu} (X)t^{2(\nu -1)} =
\prod_{j=1}^n (t^2-x_j^2)\]
and the polynomials $ p_{k,\nu}$ freely generate $\rI_{W (\fg_k,\fh_k)}(\fh_{k,\R})$.
We embed $\lsp (n,\C)$ into $\lsp (k,\C)$ by
\[\begin{pmatrix} A & B \cr C & -A^t\end{pmatrix}
\mapsto \begin{pmatrix} 0_{k-n,k-n} & 0 & 0  & 0 \cr
0& A &0 & B\cr
0 & 0 & 0_{k-n,k-n} & 0\cr
0 & C& 0 &-A^t\end{pmatrix}\]
where as usual $0$ stands for a zero matrix of the correct size.
Then
\begin{eqnarray*}
F_k(t,(0_{k-n},x))&=&t^{2k}+\sum_{\nu=1}^k p_{k, \nu} (X)t^{2(\nu -1)}\\
&=&t^{2(k-n)}\det (t+\pi_{n}(x))\\
&=&t^{2(k-n)}(t^{2n} +\sum_{\nu =1}^n p_{n,\nu }(x)t^{2(\nu -1)})\\
&=&t^{2k}+\sum_{\nu =k-n+1}^k p_{n,\nu +n-k}(x)t^{2(\nu -1)}\, .
\end{eqnarray*}
Hence
\[p_{k,\nu}|_{\fh_{n,\R}}= p_{n,\nu +n-k} \text{ for }
k-n+1 \leqq \nu \leqq k\]
and
\[p_{k,\nu}|_{\fh_{n,\R}}=0 \text{ for } 1 \leqq \nu \leqq k-n\, .\]
In particular, the restriction map
$\rI_{W (\fg_k,\fh_k)}(\fh_{k,\R})\to \rI_{W (\fg_n,\fh_n)}(\fh_{n,\R})$ is surjective.

\noindent
\textbf{The case $\mathbf{D_n}$, where $\mathbf{\fg=\so (2n,\C)}$.} We take
$\fh_{k,\R}=\R^k$ embedded in $\so (2n,\C)$ by
\begin{equation}\label{dn}
x\mapsto \begin{pmatrix} \diag(x) & 0 \cr 0 & -\diag (x)\end{pmatrix}\, .
\end{equation}
Then
$\Delta_k = \{\pm (f_i\pm f_j)\mid 1\leqq j<i\leqq k\}$ and we use the simple
root system $\Psi (\fg_k,\fh_k)= \{\alpha_1 , \dots , \alpha_k\}$ given by
\[ \alpha_1 = f_1+f_2, \text{ and } \alpha_i = f_i-f_{i-1} \text{ for }
2 \leqq i \leqq k\]
The Weyl group is
\[W (\fg_k,\fh_k)=\gS_k\rtimes
  \{\mathbf{\epsilon}\in \{1,-1\}^n\mid \epsilon_1 \cdots \epsilon_n=1\}\, .\]
In other words the elements of $W(\fg_k,\fh_k)$ contain only an
\textit{even} number of sign-changes.  The invariants are given by
\[F_k (t,X)
= t^{2k}+\sum_{\nu =2}^{k} p_{k,\nu }(X) t^{2(\nu-1)}+p_{k,1}(X)^2
= \prod_{\nu =1}^n (t^2-x_j^2)\]
where $p_1$ is the Pfaffian, $p_1(X)=(-1)^{k/2}x_1\ldots x_k$, so
$p_1(X)^2=\det (X)$.  The polynomials $p_{k,1},\ldots ,p_{k,k}$
freely generate $\rI_{W (\fg_k,\fh_k)}(\fh_{k,\R})$.

We embed $\fh_{n,\R}$ in $\fh_{k,\R}$ in the same manner as before.
This corresponds to
\[\begin{pmatrix} A & B\cr C & -A^t\end{pmatrix}\mapsto
\begin{pmatrix} 0_{k-n,k-n} & & 0_{k-n,k-n} & \cr
0 &A & 0  & B\cr
0_{k-n,k-n} & 0 & 0 \cr
0 & C & 0 & -A^t\end{pmatrix}
\, .
\]
It is then clear that
\[W_{k,n}=(\gS_{k-n}\rtimes \{1,-1\}^{k-n})\times_*
(\gS_n\rtimes \{1,-1\}^{n})\]
where the ${}_*$ indicates that $\epsilon_1\cdots \epsilon_n=1$.
Therefore, the restrictions of elements of $W_{k,n}$, $k>n$, contain
all sign changes, and
\[\gS_n\rtimes \{1,-1\}^{n-1} =
W (\fg_n,\fh_n) \subsetneqq W_{k,n}|_{\fh_{n,\R}}=\gS_n\rtimes \{1,-1\}^n\, .\]
The Pfaffian $p_{k,1}(0,X)=0$ and
\begin{eqnarray*}
F_k(t,(0,x))&=& t^{2k}+\sum_{\nu =2}^{k} p_{k,\nu }(0,x) t^{2(\nu-1)}\\
&=& t^{2(k-n)}F_n(t,x) = t^{2(k-n)}(t^{2n}+
\sum_{\nu =2}^{n} p_{n,\nu }(x) t^{2(\nu-1)} + p_{n,1}(x)^2)\\
&=& t^{2k}+\sum_{\nu=k-n+2}^k p_{n,\nu +n-k}(x)t^{2(\nu -1)}+
p_{n,1}(x)^2t^{2(k-n)}\, .
\end{eqnarray*}
Hence
$$
\begin{aligned}
&p_{k,\nu}|_{\fh_{n,\R}} = p_{n,\nu +n-k} \text{ for }
   k-n+2\leqq \nu \leqq k\, , \\
&p_{k, k-n+1}|_{\fh_{n,\R}}=p_{n,1}(x)^2\,, \text{ and } \\
&p_{k,\nu}|_{\fh_{n,\R}}=0\, ,\quad \nu =1,\ldots , k-n\, .
\end{aligned}
$$
In particular the elements in $\rI_{W(\fg_k,\fh_k)}(\fh_{k,\R})|_{\fh_{n,\R}}$ are
polynomials in even powers of $x_j$ and $p_{n,1}$ is not in the image of
the restriction map. Thus
\[\rI_{W (\fg_k,\fh_k)}(\fh_{k,\R})|_{\fh_{n,\R}}\subsetneqq
 \rI_{W (\fg_n,\fh_n)}(\fh_{n,\R})\, .\]

We put these calculations together in the following theorem.

\begin{theorem}\label{th-AdmExt}  Assume $\fg_n$ and $\fg_k$ are
simple complex Lie algebras of ranks $n$ and $k$, respectively, and
that $\fg_k$ propagates $\fg_n$.
\begin{enumerate}
\item
If $\fg_k\not= \so (2n,\C)$
then
\[
W(\fg_n,\gh_n) = W_{\gh_n}(\fg_k,\fh_k)|_{\fh_n}
= \{w|_{\fh_n} \mid w \in W(\fg_k,\fh_k) \text{ with } w(\fh_n) = \fh_n\}
\]
and the restriction map
\[\rI_{W(\fg_k,\fh_k)}(\fh_{k,\R})\to \rI_{W(\fg_n,\fh_n)}(\fh_{n,\R})\]
is surjective.
\item If $\fg_k=\so(2k,\C)$ {\rm (}so $\fg_n=\so (2n,\C)${\rm )}, then
\[W_{\gh_n}(\fg_k,\fh_k)|_{\fh_n} =\{w|_{\fh_n}\mid w
\in W(\fg_k,\fh_k) \text{ with } w(\fh_n)=\fh_n\}=\gS_n\rtimes \{1,-1\}^n\]
contains all sign changes, while the elements of $\Wn$ contain only even
numbers of sign changes. In particular
$W(\fg_n,\gh_n) \subsetneqq W_{\gh_n}(\fg_k,\fh_k)|_{\fh_n}$.  The
elements of $\rI_{W(\fg_k,\gh_k)}(\fh_{k,\R})|_{\fh_{n,\R}}$ are polynomials in the
$x_j^2$, and the Pfaffian $($square root of the determinant$)$
is not in the image of the restriction map $\rI_{W(\fg_k,\fh_k)}(\fh_{k,\R})\to
\rI_{W(\fg_n,\fh_n)}(\fh_{n,\R})$. Denote by $\rI_{W(\fg_k,\fh_k)}^{\text{even}}(\fh_{k,\R})$
the algebra of invariants that are polynomials in $x_1^2,\ldots ,x_k^2$ and
similarly for $n$. Then the restriction map $\rI_{W(\fg_k,\fh_k)}^{\text{even}}(\fh_{k,\R})
\to \rI_{W(\fg_k,\fh_k)}^{\text{even}}(\fh_{n,\R})$ is surjective.
\end{enumerate}
\end{theorem}

\begin{remark}\label{Re-RestrictionDoesNotWork}
{\rm If $\fg_k=\lsl(n,\C)$ and $\fg_n$ is constructed
from $\fg_k$ by removing any $n-k$ simple roots from the Dynkin diagram
of $\fg_k$, then part 1 of Theorem \ref{th-AdmExt} remains
valid because all the Weyl groups are permutation groups.
On the other hand, if $\fg_k$ is of type $B_k,C_k$, or $D_k$ ($k\geqq 3$)
and if $\fg_n$ is constructed from $\fg_k$ by removing at least one simple
root $\alpha_i$ with $k-i\geqq 2$,
then $\fg_n$ contains at least one
simple factor $\fl$ of type $A_\ell$, $\ell \geqq 2$. Let
$\fa$  be a Cartan subalgebra of $\fl$. Then the
restriction of the Weyl group of $\fg_k$ to
$\fa_\R$ will contain $-\mathrm{id}$.  But $-\mathrm{id}$ is not in the
Weyl group $W(\gs\gl(\ell+1,\C))$, and the restriction of the
invariants will only contain even polynomials.
Hence the conclusion of part 1 in the Theorem fails in this case.}
\hfill $\diamondsuit$
\end{remark}

We also note the following consequence of the definition of propagation.
It is implicit in the diagrams following that definition.
\begin{lemma}\label{le-RestrictionOfSimpleRoots} Assume that $\fg_k$
propagates $\fg_n$. Let $\fh_k$ be a Cartan subalgebra
of $\fg_k$ such that $\fh_n=\fh_k\cap \fg_n$ is a Cartan subalgebra of
$\fg_n$. Choose positive systems
$\Delta^+(\fg_k,\fh_k)\subset \Delta(\fg_k,\fh_k)$ and
$\Delta^+(\fg_n,\fh_n)\subset \Delta(\fg_n,\fh_n)$
such that $\Delta^+(\fg_n,\fh_n)\subseteqq
\Delta^+(\fg_k,\fh_k)|_{\fh_n}$.
Then we can number the simple roots such that
\[\alpha_{n,j}=\alpha_{k,j}|_{\fh_n}\] for $j=1,\ldots ,\dim \fh_n$.
\end{lemma}

\section{Symmetric Spaces}\label{sec3}
\noindent
In this section we discuss restriction of invariant polynomials related to
Riemannian symmetric spaces.
Let $M=G/K$ be a Riemannian symmetric space of compact or noncompact
type. Thus $G$ is a connected semisimple Lie group with an involution
$\theta$ such that
\[(G^\theta)_o\subseteqq K\subseteqq G^\theta\]
where $G^\theta =\{x\in G\mid \theta (x)=x\}$ and the subscript ${}_o$ denotes
the connected component containing the identity element. If
$G$ is simply connected then $G^\theta$ is connected and
$K=G^\theta$. If $G$ is noncompact and with finite
center, then $K\subset G$ is a \textit{maximal
compact} subgroup of $G$, $K$ is connected, and $G/K$ is simply
connected.

Denote the Lie algebra of $G$ by $\fg$. Then
$\theta$ defines an involution $\theta : \fg\to \fg$ and
$\fg=\fk\oplus \fs$
where $\fk=\{X\in\fg\mid \theta(X)=X\}$ is the Lie algebra of $K$ and
$\fs=\{X\in \fg\mid \theta (X)=-X\}$.

Cartan Duality is a bijection between the classes of simply connected
symmetric spaces of noncompact type and of compact type. On the Lie
algebra level this isomorphism is given by
$\fg = \fk\oplus \fs \leftrightarrow \fk\oplus i\fs = \fg^d$.
We denote this bijection by $M\leftrightarrow M^d$.

Fix a maximal abelian subset $\fa\subset \fs$. For $\alpha \in\fa^*_\C$ let
\[
\fg_{\C,\alpha} =\{X\in\fg_\C \mid [H,X]=\alpha (H)X \text{ for all }
H\in \fa_\C\}\, .
\]
If $\fg_{\C,\alpha}\not=\{0\}$ then $\alpha$ is called a (restricted) root. Denote
by $\Sigma (\fg,\fa)$ the set of roots.  If $M$ is of noncompact type, then
all the roots are in the real dual space $\fa^*$ and
$\fg_{\C,\alpha}=\fg_\alpha +i\fg_\alpha$, where
$\fg_\alpha =\fg_{\C,\alpha}\cap \fg$. If $M$ is of compact type, then the
roots are purely imaginary on $\fa$,
$\Sigma (\fg,\fa)\subset i\fa^*$, and $\fg_{\C,\alpha}\cap \fg=\{0\}$. The
set of roots is preserved under duality, $\Sigma(\fg,\fa )=\Sigma(\fg^d,i\fa)$,
where we view those roots as $\C$--linear functionals on $\ga_\C$.

If $\alpha \in \Sigma (\fg,\fa )$ it can happen that
$\frac{1}{2}\alpha \in\Sigma (\fg,\fa)$ or $2\alpha \in\Sigma (\fg,\fa)$. Define
\[\Sigma_{1/2}(\fg,\fa )=\{\alpha \in\Sigma(\fg,\fa ) \mid \tfrac{1}{2 }\alpha\not\in \Sigma (\fg, \fa)\}\, .\]
Then $\Sigma_{1/2}(\fg, \fa )$ is a root system in the usual sense and the Weyl
group corresponding to $\Sigma (\fg, \fa)$ is the same as the Weyl group generated
by the reflections $s_\alpha$, $\alpha \in \Sigma_{1/2}(\fg, \fa)$.
Furthermore, $M$ is
irreducible if and only if $\Sigma_{1/2}(\fg, \fa )$ is irreducible, i.e., can not be
decomposed into two mutually orthogonal root systems.

Let $\Sigma^+(\fg, \fa)\subset \Sigma (\fg, \fa)$ be a positive system and
$\Sigma^+_{1/2}(\fg,\fa )=\Sigma^+ (\fg,\fa )\cap \Sigma_{1/2} (\fg,\fa)$. Then
$\Sigma^+_{1/2}(\fg,\fa )$ is a positive root system in
$\Sigma_{1/2} (\fg,\fa)$. Denote
by $\Psi_{1/2} (\fg,\fa)$ the set of simple roots in $\Sigma_{1/2}^+(\fg,\fa)$. Then
$\Psi_{1/2} (\fg,\fa )$ is a basis for $\Sigma (\fg,\fa )$.

The list of irreducible symmetric spaces is given by the following table.
The indices $j$ and $k$ are related by $k=2j+1$.  In the fifth
column we list the realization of $K$ as a subgroup of the compact
real form.  The second column indicates the type of the root system
$\Sigma_{1/2}(\fg,\fa)$.
(More detailed information is given by the Satake--Tits diagram for $M$;
see \cite{Ar1962} or \cite[pp. 530--534]{He1978}.
In that classification the case $\SU (p,1)$, $p\geqq 1$, is denoted by $AIV$,
but here it appears in $AIII$.
The case $\SO (p,q)$, $p+q$ odd, $p\ge q>1$, is denoted by $BI$ as in
this case the Lie algebra $\fg_\C=\so (p+q,\C)$ is of type $B$.
The case $\SO (p,q)$, with $p+q$ even, $p\ge q>1$ is denoted by $DI$ as
in this case $\fg_\C$ is of type $D$. Finally, the case $\SO (p,1)$, $p $
even, is denoted by $BII$ and $\SO (p,1)$, $p$ odd, is denoted by $DII$.)

{\footnotesize
\begin{equation}\label{symmetric-case-class}
\begin{tabular}{|c|c|l|l|l|c|c|} \hline
\multicolumn{7}{| c |}
{}\\
\multicolumn{7}{| c |}
{\large{Irreducible Riemannian Symmetric $M = G/K$, $K$ connected}}\\
\multicolumn{7}{| c |}
{}\\
\hline \hline
\multicolumn{1}{|c}{} & & \multicolumn{1}{c}{$G$ noncompact}&
    \multicolumn{1}{|c}{$G$ compact} &
        \multicolumn{1}{|c}{$K$} &
        \multicolumn{1}{|c}{Rank$M$} &
        \multicolumn{1}{|c|}{Dim$M$} \\ \hline \hline
$1$ & $A_j$ &$\mathrm{SL}(j,\C)$ &$\SU (j)\times \SU(j)$ & $\diag\, \SU(j)$ & $j-1$ & $j^2-1$ \\ \hline
$2$ & $B_j$&$\SO (k,\C)$&  $\SO (k)\times \SO (k)$ & $\diag\, \SO (k)$ &
    $j$ & $2j^2+j$ \\ \hline
$3$ & $D_j$&$\SO (2j,\C)$ & $\SO (2j)\times \SO (2j)$ & $\diag\, \SO(2j)$ &
    $j$ & $2j^2-j$ \\ \hline
$4$ & $C_j$&$\Sp (j,\C)$&$\Sp (j)\times \Sp (j)$ & $\diag\,\Sp (j)$ & $j$ & $2j^2+j$ \\ \hline
$5$ & $AIII$& $\SU (p,q)$&$\SU(p+q)$ & $\mathrm{S}(\U (p)\times \U ( q))$ &
    $\min(p,q)$ & $2pq$ \\ \hline
$6$ &$AI$ &$\mathrm{SL}(j,\R)$& $\SU (j)$ & $\SO (j)$ & $j-1$ & $\tfrac{(j-1)(j+2)}{2}$ \\ \hline
$7$ &$AII$ &$\SU^*(2j)$& $\SU (2j)$ & $\Sp (j)$ & $j-1$ & $2j^2-j-1$  \\ \hline
$8$ &$BDI$&$\SO_o (p,q)$ &$\SO (p+q)$ & $\SO (p) \times \SO (q)$ &
    $\min(p,q)$ & $pq$  \\ \hline
$9$ &$DIII$&$\SO^*(2j)$ &$\SO (2j)$ & $\U (j)$ & $[\tfrac{j}{2}]$ & $j(j-1)$ \\ \hline
$10$ &$CII$&$\Sp(p,q)$ &$\Sp (p+q)$ & $\Sp (p) \times \Sp (q)$ &
    $\min(p,q)$ & $4pq$  \\ \hline
$11$ & $CI$& $\Sp (j,\R)$  &$\Sp (j)$ & $\U (j)$ & $j$ & $j(j+1)$  \\ \hline
\end{tabular}
\end{equation}
}

Only in the following cases do we have
$\Sigma_{1/2}(\fg,\fa )\not= \Sigma (\fg,\fa)$:
\begin{itemize}
\item $AIII$ for $1 \leqq p < q$,
\item $CII$ for $1 \leqq p < q$, and
\item $DIII$ for $j$ odd.
\end{itemize}
In those three cases there is exactly one simple root with
$2\alpha \in\Sigma (\fg,\fa )$
and this simple root is  at the
right end of the Dynkin diagram for $\Psi_{1/2} (\fg,\fa )$. Also, either
$\Psi_{1/2} (\fg,\fa )=\{\alpha\}$ contains one simple root or
$\Psi_{1/2}(\fg,\fa )$ is of type $B_r$
where $r=\dim \fa$ is the rank of $M$.

Finally, the only two cases where $\Psi_{1/2} (\fg,\fa )$ is of type $D$ are
the case $\SO (2j,\C)/\SO(2j)$ or  the split case
$\SO_o(p,p)/\SO (p)\times \SO (p)$.

Later on we will also need the root system $\Sigma_2 (\fg,\fa )=
\{\alpha \in \Sigma (\fg,\fa )\mid 2\alpha\not\in \Sigma (\fg,\fa )\}$.
According to the above discussion, this will only change the simple root at the
right end of the Dynkin diagram. If $\Psi_2(\fg,\fa )$
is of type $B$ the root system $\Sigma_2 (\fg,\fa )$ will be of type $C$.

Let $G/K$ be an irreducible symmetric space of compact or non-compact type. As
before let $\fa\subset \fs$ be maximal abelian. Let $\fh$ be a Cartan subalgebra of
$\fg$ containing $\fa$. Then $\fh=(\fh\cap \fk) \oplus \fa$. Let
$\Delta (\fg,\fh)$, $\Sigma (\fg,\fa)$, and $\Sigma_{1/2}(\fg,\fa)$ denote the corresponding
root systems and $W (\fg,\fh)$ respectively $W (\fg,\fa)$ the Weyl group
corresponding to $\Delta (\fg,\fh)$ respectively $\Sigma (\fg,\fa)$. We define
an extension of those Weyl groups $\widetilde{W}(\fg,\fh)$ and $\widetilde{W}(\fg,\fa)$ in the following way:
If the root system in question (i.e., $\Delta (\fg,\fh)$ or $\Sigma_{1/2}(\fg,\fa)$)  is not of type $D$ then
$\widetilde{W}$ is just the Weyl group. If the root system is of type $D$,
so the Weyl group elements involve only
even numbers of sign changes, then $\widetilde{W}$
is the $\Z_2$--extension of the Weyl group allowing all sign changes.
Denote
$\widetilde{W}_{\fa}(\fg,\fh)=\{w\in \widetilde{W}(\fg,\fa)\mid w(\fa )=\fa\}$.

Note $\widetilde{W}(\fg,\fa)\not=W (\fg,\fa)$  only for
$M$ locally isomorphic to
$\SO (2j,\C)/\SO (2j)$ (where $\fh = \fa_\C$) or its compact dual
$(SO(2j)\times SO(2j))/\diag\, SO(2j)$ (where $\fh \cong \fa \oplus \fa$),
or to $\SO_o(j,j)/\SO (j)\times \SO(j)$
or its compact dual $\SO (2j)/\SO (j)\times \SO (j)$ where $\fh = \fa$.

If $G/K$ is reducible without Euclidean factors
then the Weyl groups are direct products of Weyl groups
for the irreducible factors. Then $\widetilde{W}(\fg,\fh)$ and
$\widetilde{W}(\fg,\fa)$ denote the corresponding products of the extended Weyl
groups for each irreducible factor.

\begin{theorem}\label{th-IhIa} Let $G/K$ be a symmetric space of
compact or non-compact type (thus no Euclidean factors).
In the above notation,
$\widetilde{W}(\fg,\fa)= \widetilde{W}_\fa (\fg,\fh)|_{\fa}$ and
the restriction map
$I_{\widetilde{W}(\fg,\fh)}(\fh_\R )\to I_{\widetilde{W}(\fg,\fa)}(\fa )$ is
surjective.
\end{theorem}
\begin{proof} We can assume that $G/K$ is irreducible. If neither
$\Delta (\fg,\fh)$ nor $\Sigma (\fg,\fa)$ is of type $D$ this is
Theorem 5 from \cite{He1964}. According to the above discussion, the only
cases where $\Sigma (\fg,\fa)$ is of type $D$ are where $\Delta (\fg,\fh)$
is also of type $D$ and $\fa = \fh_\R$, or $\fa$ is the diagonal in
$\fh \cong \fa \oplus \fa$, or $\fa = \fh$.  The statement is clear when
$\fa$ is $\fh$ or $\fh_\R$.  If $\fa$ is the diagonal in $\fh \cong
\fa \oplus \fa$ then $\widetilde{W}_\fa (\fg,\fh)$ is the diagonal in
$\widetilde{W}(\fg,\gh) \cong
\widetilde{W}(\fg,\fa) \times \widetilde{W}(\fg,\fa)$, hence again is
$\widetilde{W}(\fg,\fa)$.

Now suppose that neither $\Delta (\fg,\fh)$ nor $\Sigma_{1/2} (\fg,\fa)$
is of type $D$.  Then $\widetilde{W}(\fg,\fa)=W(\fg,\fa)$ consists of all
permutations with sign changes (with respect to the correct basis). The
claim now follows from  the explicit calculations in
\cite[pp. 594, 596]{He1964}.
\end{proof}

Let $M_k=G_k/K_k$ and $M_n=G_n/K_n$ be irreducible
symmetric spaces of compact or noncompact type.  We say that
$M_k$ \textit{propagates} $M_n$, if $G_n\subseteqq G_k$, $K_n=K_k\cap G_n$,
and either $\fa_k=\fa_n$ or choosing $\fa_n\subseteqq \fa_k$ we
only add simple roots to the left end
of the Dynkin diagram for $\Psi_{1/2}(\fg_n,\fa_n)$ to obtain the Dynkin diagram
for $\Psi_{1/2}(\fg_k,\fa_k)$.
So, in particular $\Psi_{1/2}(\fg_n,\fa_n)$ and
$\Psi_{1/2}(\fg_k, \fa_k)$ are of the same type.
In general, if
$M_k$ and $M_n$ are Riemannian symmetric spaces
of compact or noncompact type, with universal covering
$\widetilde{M_k}$ respectively $\widetilde{M_n}$, then $M_k$
\textit{propagates}
$M_n$ if we can enumerate the irreducible factors of $\widetilde{M}_k= M_k^1\times
\ldots \times M_k^j$ and $\widetilde{M}_n=M_n^1\times \ldots \times M_n^i$, $i \leqq j$
so that $M_k^s$ propagates $M_n^s$ for $s=1,\ldots ,i$.
Thus, each $M_n$ is, up to covering, a product of irreducible factors
listed in Table \ref{symmetric-case-class}.

In general we can construct infinite sequences of propagations by moving
along each row in Table \ref{symmetric-case-class}. But there are also  inclusions like
$\mathrm{SL} (n,\R )/\SO (n)\subset \mathrm{SL} (k,\C)/\SU (k)$ which
satisfy the definition of propagation.

When $\fg_k$ propagates $\fg_n$, and $\theta_k$ and
$\theta_n$ are the corresponding involutions with
$\theta_k|_{\fg_n} = \theta_n$, the corresponding eigenspace decompositions
$\fg_k=\fk_k\oplus \fs_k$ and $\fg_n=\fk_n\oplus \fs_n$
give us
\[
\fk_n=\fk_k\cap \fg_n\, ,\quad \text{and}\quad \fs_n=\fg_n\cap \fs_k\, .\]
We recursively choose maximal commutative subspaces $\fa_k\subset \fs_k$ such
that $\fa_{n} \subseteqq \fa_k$ for $k\geqq n$.
Denote by $W(\fg_n,\fa_n)$ and $W(\fg_k,\fa_k)$ the corresponding Weyl
groups. The extensions $\widetilde{W}(\fg_k,\fa_k)$ and
$\widetilde{W}(\fg_n,\fa_n)$ are defined as just before Theorem \ref{th-IhIa}.
Let  $\rI(\fa_n)=\rI_{W (\fg_n,\fa_n)}(\fa_n)$,
$\rI_{\widetilde W (\fg_n,\fa_n)}(\fa_n)$, and
$\rI_{\widetilde W (\fg_k,\fa_k)}(\fa_k)$ denote the respective sets of Weyl
group invariant or $\widetilde{W}$--invariant polynomials on $\fa_n$ and
$\fa_k$. As before we let
\begin{equation}\label{eq-WknA}
W_{\fa_n}(\fg_k,\fa_k):=\{w\in W (\fg_k,\fa_k) \mid w(\fa_n)=\fa_n\}
\end{equation}
and define $\widetilde{W}_{\fa_n}(\fg_k,\fa_k)$ in the same way.

\begin{theorem}\label{th-AdmExtG/K} Assume that $M_k$ and $M_n$ are
symmetric spaces of compact or noncompact type and that $M_k$ propagates
$M_n$.

\noindent {\rm (1)} If $M_n$ does not contain any irreducible factor
with $\Psi_{1/2} (\fg_n,\fa_n)$ of type $D$, then
\begin{equation}\label{eq-RestrictionOfWeyl1}
W_{\fa_n}(\fg_k,\fa_k)|_{\fa_n} = W(\fg_n,\fa_n)
\end{equation}
and the restriction map $\rI(\fa_k)\to \rI(\fa_n)$ is surjective.

\noindent {\rm (2)} If $\Psi_{1/2}(\fg_n,\fa_n)$ is of type $D$ then
$$W (\fg_n,\fa_n)\subsetneqq
 W_{\fa_n} (\fg_k,\fa_k)|_{\fa_n} \text{ and }
\rI_{W(\fg_k,\fa_k)}(\fa_k)|_{\fa_n}\subsetneqq
  \rI_{W(\fg_n,\fa_n)}(\fa_{n}) .$$
On the other hand
$\widetilde{W}(\fg_n,\fa_n)= \widetilde{W}_{\fa_n} (\fg_k,\fa_k)|_{\fa_n}$
and $\rI_{\widetilde{W}(\fg_n,\fa_n)}(\fa_n) = \rI_{\widetilde{W}(\fg_k,\fa_k)}(\fa_k)|_{\ga_n}$.

\noindent {\rm (3)} In all cases
$\widetilde{W}(\fg_n,\fa_n)= \widetilde{W}_{\fa_n} (\fg_k,\fa_k)|_{\fa_n}$
and $\rI_{\widetilde{W} (\fg_k,\fa_k)}
(\fa_k)|_{\fa_n}= \rI_{\widetilde{W}(\fg_n,\fa_n)}(\fa_n)$.
\end{theorem}

\begin{proof} It suffices to prove this for each irreducible component of
$M_n$.  The argument of Theorem \ref{th-AdmExt} is valid here as well, and
our assertion follows.
\end{proof}

\section{Application to Fourier Analysis on Symmetric Spaces of the noncompact
Type}\label{sec4}
\noindent
In this section we apply the above results to harmonic
analysis on symmetric spaces.
We start by recalling the main ingredients for the Helgason Fourier transform
on a Riemannian symmetric space $G/K$ of the noncompact type. The material
is standard and we refer to the books of
Helgason, in particular \cite{He1984}, for details.  We use the
notation from the previous section: $\Sigma (\fg,\fa )$
is the set of (restricted) roots of $\fa$ in $\fg$ and
$\Sigma^+(\fg,\fa )\subset \Sigma (\fg,\fa )$ is a positive system. Let
$$
\fn = \bigoplus_{\alpha \in \Sigma^+(\fg,\fa )} \fg_\alpha, \ \ \
\fm = \fz_{\fk}(\fa), \ \ \text{ and } \ \
\gp = \fm + \fa + \fn.
$$
Denote by $N$ (respectively $A$) the analytic subgroup of
$G$ with Lie algebra $\fn$ (respectively $\fa$). Let
$M=Z_K(\fa)$ and $P=MAN$. Then $M$ and $P$ are
closed subgroup of $G$ and $P$ is a \textit{minimal parabolic subgroup}.
Note, that we are using $M$ in two different ways, once as the
symmetric space $G/K$ and also as a subgroup of $G$. The meaning
will always be clear from the context.

We have the Iwasawa decomposition
\[G=NAK:\ C^\omega\text{--diffeomorphic to } N\times A\times K \text{ under }
(n,a,k) \mapsto nak\, .\]
For $x\in G$ denote by $a(x)\in A$ the unique element in $A$ such that
$x\in Na(x)K$. Then $x\mapsto a(x)$ is analytic. For $\lambda\in \fa_\C^*$ let
\[a^\lambda :=e^{\lambda (\log (a))}\, .\]
Then the characters on $A$ are given by
\[a\mapsto \chi_\lambda (a):=a^\lambda \]
for some $\lambda\in\fa_\C^*$. $\chi_\lambda$ is unitary if and only if
$\lambda \in i\fa^*$. Let $m_\alpha =\dim \fg_\alpha$ and
\[\rho = \tfrac{1}{2}\sum_{\alpha\in\Sigma^+(\fg,\fa)}m_\alpha \, \alpha\, .\]

For a moment let $G$ be a locally compact topological group and
$K\subset G$ a compact subgroup
A continuous, non-zero, $K$--biinvariant function $\varphi : G\to\C$ is
\textit{$K$-spherical}
or just \textit{spherical} if for all $x,y\in G$ we have
$$\int_{K} \varphi (xky)\, dk=\varphi (x)\varphi (y)\,.$$
We will view the spherical functions on $G$ as $K$--invariant functions
on $G/K$. The importance of the spherical functions comes from the
fact that a map $\chi : L^1(G/K)^{K}\to \C$ is a continuous algebra
homomorphism if and only if there exists a bounded spherical function $\varphi$
such that $\chi (f)=\int_{G/K} f(x)\overline{\varphi (x)}\, dx$.
The spherical function $\varphi$ is called \textit{positive definite} if
for all $n\in \N$, $c_j\in \C$, $x_j\in G$, $1\leqq j\leqq n$ we
have
$$\sum_{\nu,\mu=1}^n c_\nu \overline{c_\mu }\varphi (x_\nu^{-1}x_\mu )\geqq 0\, ,$$
in other words, if the matrix $(\varphi (x_\nu^{-1}x_\mu))_{\nu,\mu}$ is
positive semidefinite for finite subsets $\{x_1,\ldots ,x_n\}$ of
$G$. The positive definite spherical functions are particular coefficient
functions
\begin{equation}\label{def-spherical-n}
\varphi (g) = ( e_{\pi },\pi (g)e_{\pi } )
\end{equation}
where $\pi$ is an irreducible unitary representation with
nonzero $K$--fixed vectors and $e_{\pi } \in V_\pi$ is  $K$--fixed unit vector.

For $\lambda \in \fa_\C^*$ define
\begin{equation}\label{def-spherical}
\varphi_\lambda (x)=\int_K \chi_{\lambda -\rho}(a(kx))\, dk
\end{equation}
where the Haar measure $dk$ on $K$ is normalized by $\int_K\, dk=1$.
Then $\varphi_\lambda$ is a \textit{spherical function} on $G$, $\varphi_\lambda
=\varphi_\mu$ if and only if $\mu\in W (\fg,\fa)^\prime \cdot \lambda$,
and every spherical function is equal to some $\varphi_\lambda$.
Here ${}^\prime$ stands
for the transpose of the elements in $W(\fg,\fa)$ acting on
$\fa^*$ and $\fa_\C^*$  The function
$\varphi_\lambda$ is positive definite when $\lambda\in i\fa^*$.

The \textit{spherical Fourier transform} on $M$ is given by
\[\cF(f)(\lambda)=\widehat{f}(\lambda ):=\int_M f(x)\varphi_{-\lambda }(x)
\, dx\, \quad f\in C_c^\infty (M)^K\, .\]
The invariant measure $dx$ on $M$ can be normalized so that
the spherical Fourier transform extends to an unitary isomorphism
\[f\mapsto \widehat{f}\, ,\quad L^2(M)^K\cong L^2(i\fa^*_+,
|c( \lambda )|^{-2}d\lambda )
\cong
L^2\left(i\fa^*, \tfrac{d\lambda}{\# W (\fg,\fa ) |c(\lambda )|^2}\right)^{W(\fg,\fa)}\]
where $\fa_+^* = \{\lambda \in \fa^* \mid \langle \lambda,
\alpha \rangle > 0 \text{ for all } \alpha \in \Sigma^+(\fg,\fa)\}$ and
$c(\lambda )$ denotes the Harish-Chandra $c$--function.
For $f\in C_c^\infty (M)^K $ the inversion is given by
\[f(x)=\frac{1}{\# W (\fg,\fa )} \int_{i\fa^*} \widehat{f}(\lambda )\varphi_{\lambda }
(x)\frac{d\lambda }{|c (\lambda )|^2}\, .\]

If $\sigma :\fa\to \fa$
is linear, then $\sigma':\fa^*_\C\to \fa^*_\C$ is its transpose,
$\sigma'(\lambda )(H)=\lambda (\sigma (H))$.

Recall the notation from the last section. If
$\Sigma (\fg,\fa)=\Sigma_{1/2}(\fg,\fa)$
then $\Psi (\fg,\fa)$ is the simple root system.

A connected semisimple Lie group $G$ is {\it algebraically simply connected}
if it is an analytic subgroup of the connected simply connected group
$G_\C$ with Lie algebra $\fg_\C$.  Then the analytic subgroup $K$ of $G$
for $\gk$ is compact, and every automorphism of $\fg$ integrates to an
automorphism of $G$.

\begin{lemma}\label{le-AutDynkDiagrG} Let $G/K$ be a Riemannian
symmetric space of noncompact type with $G$ simple and algebraically
simply connected.  Suppose that $\fa$ is a Cartan subalgebra of $\fg$,
i.e., that $\fg$ is a split real form of $\fg_\C$. If $\sigma :\fa\to \fa$ is
a linear isomorphism such that $\sigma'$ defines an automorphism of
the Dynkin diagram of $\Psi (\fg,\fa)$, then there exists a
automorphism $\widetilde{\sigma} :G\to G$ such that
\begin{enumerate}
\item $\widetilde{\sigma}|_\fa=\sigma$ where by abuse of notation we write
$\widetilde{\sigma}$ for $d\widetilde{\sigma}$,
\item $\widetilde{\sigma}$ commutes with the
the Cartan involution $\theta$, and in particular
$\widetilde{\sigma}(K)=K$,
\item $\widetilde{\sigma}(N)=N$.
\end{enumerate}
\end{lemma}

\begin{proof} The complexification of $\fa$ is a Cartan subalgebra
$\fh$ in $\fg_\C$ such that $\fh_\R=\fa$. Let
$\{Z_\alpha\}_{\alpha\in\Sigma (\fg,\fa)}$ be a
Weyl basis for $\fg_\C$ (see, for example, \cite[page 285]{V1974}).
Then (see, for example, \cite[Theorem 4.3.26]{V1974}),
\[\fg_0=\fa \oplus \bigoplus_{\alpha\in \Delta (\fg,\fh)}\R Z_\alpha
\]
is a real form of $\fg_\C$. Denote by $B$ the Killing form
of $\fg_\C$. Then $B(Z_\alpha ,Z_{-\alpha})=-1$ and
it follows that $B$ is positive definite on $\fa$ and on
$\bigoplus_{\alpha\in\Sigma^+(\fg,\fa)}\R (Z_\alpha- Z_{-\alpha})$,
and negative definite on
$\bigoplus_{\alpha\in\Sigma^+(\fg,\fa)}\R (Z_\alpha + Z_{-\alpha})$.
Hence, the map
$$\theta|_{\fa }=- \id \text{ and } \theta(Z_\alpha) = Z_{-\alpha}$$
defines a Cartan involution on $\fg_0$ such that the Cartan subalgebra
$\fa$ is contained in the corresponding $-1$ eigenspace $\fs$. As there
is (up to isomorphism) only one real form of $\fg_\C$ with Cartan
involution
such that $\fa\subset \fs$ we can assume
that $\fg=\fg_0$ and the above Cartan involution $\theta$ is the
the one we started with.

Going back to the proof of \cite[Lemma 4.3.24]{V1974}
the map defined by
\[\widetilde{\sigma}|_\fa = \sigma \, \quad \text{and}
\quad \widetilde{\sigma}(Z_\alpha ) =Z_{\sigma {\alpha}}\]
is a Lie algebra isomorphism $\widetilde{\sigma}:\fg\to \fg$.
But then
$$
\widetilde{\sigma}(\theta (Z_\alpha)) = \widetilde{\sigma}(Z_{-\alpha})
= Z_{\sigma (-\alpha )} = Z_{-\sigma (\alpha )} =
\theta (\widetilde{\sigma }(Z_\alpha)).
$$
Finally, $\theta|_\fa = -\id$ and it follows that $\widetilde{\sigma }$
and $\theta$ commute. As
\[\fk =\bigoplus_{\alpha\in\Sigma^+(\fg,\fa)}\R (Z_\alpha +\theta
(Z_{\alpha}))\]
and $\sigma (\Sigma^+(\fg,\fa))=\Sigma^+(\fg,\fa)$ it follows that
$\widetilde{\sigma} (\fk )=\fk$.

As $\sigma'(\Sigma^+(\fg,\fa))=\Sigma^+(\fg,\fa)$ it follows that
$\widetilde{\sigma}(\fn)=\fn$.

As $G$ is assumed to be algebraically simply connected, there is an
automorphism of $G$ with differential $\widetilde{\sigma}$. Denote this
automorphism also by $\widetilde{\sigma}$. It is clear that
$\widetilde{\sigma}$ satisfies the assertions of the Lemma.
\end{proof}

\begin{theorem}\label{th-SphericalFctTildeWInv} Let $G/K$ be a
Riemannian symmetric space of noncompact type with $G$ simple and
algebraically simply connected.  Let $\fa$ be a Cartan subalgebra
of $\fg$ and $\sigma :\fa\to \fa$ a linear isomorphism such
that $\sigma'$ defines an automorphism of the Dynkin diagram of
$\Psi(\fg,\fa)$.  Then for $x\in G$
\[\varphi_\lambda ( \widetilde{\sigma}(x))=\varphi_{\sigma'(\lambda
)}(x)\, .\]
If $f\in L^2(G/K)^K$ is such that $f$ is $\widetilde{\sigma}$-invariant,
then
$\cF (f)$ is $\sigma'$-invariant.
\end{theorem}

\begin{proof} Let $\widetilde{\sigma}: G\to G$ be the automorphism of
from Lemma \ref{le-AutDynkDiagrG}.
As $\widetilde{\sigma}(K)=K$, $\widetilde{\sigma}(A)=A$,
and $\widetilde{\sigma}(N)=N$, it follows that if $x=n(x)a(x)k(x)$
is the
Iwasawa decomposition of $x\in G$, then
\[\widetilde{\sigma}(x)=\widetilde{\sigma}(n(x))
\widetilde{\sigma}(a(x))\widetilde{\sigma}(k(x))\]
is the corresponding decomposition of $\widetilde{\sigma}(x)$. By
(\ref{def-spherical}) and the fact that $\sigma'(\rho)=\rho$
we get
for $x\in G$:
\begin{eqnarray*}
\varphi_\lambda (\widetilde{\sigma}(x))&=&
\int_K \chi_{\lambda-\rho}(a(k\widetilde{\sigma}(x)))\, dk\\
&=&\int_K \chi_{\lambda-\rho}(\widetilde{\sigma}( a(kx))) \, dk\\
&=&\int_K \chi_{\sigma'(\lambda-\rho)}(a(kx))\, dk
=\varphi_{\sigma'(\lambda )}(x)
\end{eqnarray*}
where we have used that the Haar measure on $K$ is invariant under
$\widetilde{\sigma}|_K$.

The first statement follows now by applying this to $x=\exp (H)$,
$H\in\fa$.  Using $f\circ \widetilde{\sigma}=f$, the
second statement follows because
$G$-invariant measure on $G/K$ is is $\widetilde{\sigma}$-invariant
and we can assume that $f\in C_c^\infty (G/K)$.
\end{proof}

Fix a positive definite $K$--invariant bilinear form
$\langle \cdot ,\cdot \rangle $ on $\fs$. It defines an
invariant Riemannian structure on $M$ and hence also an invariant
metric $d(x,y)$. Let $x_o=eK\in M$ and for $r>0$ denote by $B_r=B_r(x_o)$
the closed ball
\[B_r =\{x\in M\mid d(x,x_o)\leqq r\}\, .\]
Note that $B_r$ is $K$--invariant. Denote by $C_r^\infty (M)^K$ the space of
smooth $K$--invariant functions on $M$ with support in $B_r$.
The restriction map $f\mapsto f|_A$ is a bijection from
$C_r^\infty (M)^K$ onto $C_r^\infty (A)^{W(\fg,\fa )}$ (using the
obvious notation).  Define
$\widetilde{W}(\fg,\fa)$  as
just before Theorem \ref{th-AdmExtG/K}. Let
$C_{r,\widetilde{W}(\fg,\fa)}^\infty (M)^K$ be the preimage
of $C_r^\infty (A)^{\widetilde{W}(\fg,\fa)}$ in $C_r^\infty (M)^K$.
In particular $C_{r,\widetilde{W}(\fg,\fa)}^\infty (M)^K$ is just $C^\infty_r(M)^K$
if there is no irreducible factor for which $\Psi_{1/2}(\fg,\fa)$ is of
type $D$.
In the case $\Psi_{1/2}(\fg,\fa)$ is of type $D$ we can replace $G$ by
$G_{\widetilde{\sigma}}=G\ltimes \{1,\widetilde{\sigma}\}$ and $K$ by
$K_{\widetilde{\sigma}}=K\ltimes \{1,\widetilde{\sigma}\}$.
Then $M=G_{\widetilde{\sigma}}/K_{\widetilde{\sigma}}$ and
$C_{r,\widetilde{W}(\fg,\fa)}^\infty (M)^K=C_{r}^\infty (M)^{K_{\widetilde{\sigma}}}$.
This corresponds to replacing $\SO (2j,\C)$ by $\mathrm{O}(2j,\C)$. In that
case we choose
\[\fa =\left\{\left.\begin{pmatrix} t_1 X & & \\
&\ddots & \\
& & t_n X\end{pmatrix}\, \right| t_1,\ldots ,t_n\in\R\right\}\text{ where } X=
\begin{pmatrix} 0 & 1\\ -1 & 0\end{pmatrix},\] and then
then $\widetilde{\sigma}$ is conjugation by $\diag (1,\ldots ,1, -1)$.

Denote by $\PW_r (\fa^*_\C)$ the Paley-Wiener space on $\fa_\C^*$ and by
$\PW_r(\fa^*_\C)^{\widetilde{W}(\fg,\fa)}$ the space of $\widetilde{W}(\fg,\fa)$--invariant functions
in $\PW_r(\fa^*_\C)$.

The following is a simple modification
of the Paley-Wiener theorem of Helgason \cite{He1966,He1984} and Gangolli
\cite{Ga1971}; see \cite{OP2004} for a short overview.

\begin{theorem}[The Paley-Wiener Theorem]\label{th-IsomorphismNonCompact1} The Fourier
transform defines  bijections
$$
 C^\infty_r (M)^K \cong \PW_r(\fa_\C)^{W(\fg,\fa) } \text{ and }
C_{r,\widetilde{W}(\fg,\fa)}^\infty (M)^K \cong
\PW_r(\fa_\C )^{\widetilde{W}(\fg,\fa)}\, .
$$
\end{theorem}
\begin{proof} For the proof we need only consider factors of type
$\SO (2j,\C)/\SO (2j)$ and of type $\SO (j,j)/\SO (j)\times \SO (j)$.
It is enough to show that in both cases we have
\[\varphi_\lambda (\exp (w(H)))=\varphi_{w^\prime \lambda}(\exp H)\]
for all $w\in\widetilde{W}(\fg,\fa)$. This is well
known for $w\in W (\fg,\fa)$. Note that in both
cases the root system $\Sigma (\fg,\fa)=\Delta (\fg,\fh)$ is of type $D$.

In the first case, the spherical function is given (see \cite{He1966},
p. 432) by
\[\varphi_\lambda (a)=
\frac{\varpi (\rho )}{\varpi (\lambda )}
\frac{\sum_{w\in W(\fg,\fa )}(\det w)a^{w^\prime \lambda }
}{\prod_{\alpha\in\Sigma^+(\fg,\fa )}(a^{\alpha/2}-a^{-\alpha/2})}\,\,
\text{ where }\,\,
\varpi (\lambda )=\prod_{\alpha\in\Delta^+} \langle\lambda ,\alpha
\rangle\, ,\]
and the claim follows by direct calculation.

For the second case
we recall first that the Weyl group contains all \textit{even} sign
changes
and permutations, whereas
$\widetilde{W}(\fg,\fa)$ contains all sign changes and the permutations.
Thus, fixing one element $\sigma \in\widetilde{W}(\fg,\fa)\setminus
W(\fg,\fh)$
we have
\[\widetilde{W}(\fg,\fa)= W(\fg,\fa)\cup \sigma W(\fg,\fa)
\]
Using the notation for the root system $D_j$ in Section \ref{sec2} we
take
\[\sigma (f_1)=-f_1\quad\text{and}\quad \sigma (f_k)=f_k\, ,\quad
k=2,\ldots ,j\, .\]
Then $\sigma (\alpha_1)=\alpha_2$, $\sigma (\alpha_2)=\alpha_1$ and
$\sigma (\alpha_i)
=\alpha_i$ for all $i\ge 3$. Thus $\sigma$ defines an automorphism of
the Dynkin diagram for $\Psi (\fg,\fa)$. The statement follows now from
Theorem \ref{th-SphericalFctTildeWInv}.
\end{proof}
We assume now that $M_k$ propagates $M_n$, $k\geqq n$. The index $j$ refers
to the symmetric space $M_j$, for a function $F$ on $\fa_{k,\C}^*$ let $P_{k,n}(F):=
F|_{\fa_{n,\C}}$. We fix a compatible $K$--invariant
inner products on $\fs_n$ and $\fs_k$, i.e., for all $X,Y\in\fs_n\subseteqq \fs_k$ we have
\[\langle X,Y\rangle_k=\langle X,Y\rangle_n\, .\]
\begin{theorem}[Paley-Wiener Isomorphisms]\label{th-IsomorphismNonCompact2}
Assume that $M_k$ propagates $M_n$. Let $r>0$. Then the following holds:
\begin{enumerate}
\item The map $P_{k,n}:
\PW_r (\fa_{k,\C}^*)^{\widetilde{W}(\fg_k,\fa_k)} \to
\PW_r(\fa_{n,\C}^*)^{\widetilde{W}(\fg_n,\fa_n)}$ is surjective.
\item The map $C_{k,n}=\cF^{-1}_n\circ P_{k,n}\circ \cF_k:C^\infty_{r,\widetilde{W}(\fg_k,\fa_k)}
(M_k)^{K_k}
\to C^\infty_{r,\widetilde{W}(\fg_n,\fa_n)}(M_n)^{K_n}$ is surjective.
\end{enumerate}
\end{theorem}
\begin{proof} This follows from Theorem \ref{th-IsoPW1} and Theorem \ref{th-IsomorphismNonCompact1}.
\end{proof}

We assume now that $\{M_n,\iota_{k,n}\}$ is a injective system of symmetric spaces
such that $M_k$ is a propagation of $M_n$. Here $\iota_{k,n}: M_n\to M_k$ is the
injection. Let
\[M_\infty =\varinjlim M_n\, .\]

We have also, in a natural way, injective systems
$\fg_n\hookrightarrow \fg_k$, $\fk_n\hookrightarrow \fk_k$, $\fs_n\hookrightarrow \fs_k$,
and $\fa_n\hookrightarrow \fa_k$ giving rise to corresponding injective systems. Let
$$
\mfg_\infty :=\varinjlim \mfg_n\, ,\quad \gk_\infty :=\varinjlim \gk_n\, ,\quad
\gs_\infty :=\varinjlim\gs_n\, , \quad \ga_\infty := \varinjlim\ga_n\, ,
 \quad\text{and} \quad \fh_\infty :=\varinjlim \fh_n\, .$$
Then $\mfg_\infty=\gk_\infty \oplus \gs_\infty$ is the eigenspace decomposition of
$\mfg_\infty $ with respect to the involution $\theta_\infty :=\varinjlim \theta_n$,
$\fa_\infty$ is a maximal abelian subspace of $\fs_\infty$, and $\fh_\infty$ is a Cartan subalgebra
of
$\fg_\infty$.

The restriction maps $\res_{k,n}$ and the
maps from Theorem \ref{th-IsomorphismNonCompact2} define projective systems
$\{I_{\widetilde{W} (\fg_n,\fa_n)}(\fa_n)\}_n$,
$\{\PW_{r}(\fa_{n,\C})^{\widetilde{W}(\fg_n,\fa_n)}\}_n$, and
$\{C_{r,\widetilde{W}(\fg_n,\fa_n)}(M_n)^{K_n}\}_n$.
Let
\begin{eqnarray*}
I_{\infty }(\fa_\infty )&:=&\varprojlim I_{\widetilde{W}(\fg_n,\fa_n)}(\fa_n)\\
\PW_r(\fa^{*}_{\infty, \C} )&:=&\varprojlim \PW_{r}(\fa_{n,\C}^*)^{\widetilde{W} (\fg_n,\fa_n)}\\
C_{r,\infty}^\infty (M_\infty)^{K_\infty}
&:=&\varprojlim C^\infty_{r,\widetilde{W}(\fg_n,\fa_n)}(M_n)^{K_n} \, .
\end{eqnarray*}

We note, that by the explicit construction in Section \ref{sec2},
there is a canonical inclusion $\widetilde{W}(\fg_n,\fa_n)
\stackrel{\iota_{k,n}}{\hookrightarrow}\widetilde{W}_{\fa_n}(\fg_k,\fa_k)$ such
that $\iota_{k,n}(s)|_{\fa_n}=s$. In this way, we get an injective
system $\{\widetilde{W}(\fg_n,\fa_n)\}_n$. Let
\[\widetilde W_\infty =\varinjlim \widetilde{W}(\fg_n,\fa_n)\, .\]
Then we can view $I_{\widetilde{W}_\infty }(\fa_\infty )$ as $\widetilde{W}_\infty$--invariant
polynomials on $\fa_\C^{\infty *}$ and $\PW_r(\fa^{*}_{\infty , \C })$ as
$\widetilde{W}_\infty$--invariant functions on $\fa_\C^{\infty *}$.
The projective limit $C_{r,\infty}^\infty (M_\infty)^{K_\infty}$ consists of
functions on on $A_\infty =\varinjlim A_n$, where $A_n=\exp \fa_n$.  In
Section \ref{sec8} we discuss a direct limit function space on $M_\infty$
that is more closely related to the representation theory of $G_\infty$.

For $\mathbf{f}=(f_n)_n\in
C_{r,\infty}^\infty (M_\infty)^{K_\infty}$ define $\cF_\infty (\mathbf{f})
\in
\PW_r(\fa^{*}_{\infty ,\C} )$ by
\begin{equation}
\cF_\infty (\mathbf{f}):=\{\cF_n(f_n)\}\, .
\end{equation}
Simplify the notation by setting $\widetilde{W}_n=\widetilde{W}(\fg_n,\fa_n)$.
Then $\cF_\infty (\mathbf{f})$ is well defined by Theorem \ref{th-IsomorphismNonCompact2}
and we have a commutative diagram
$$\xymatrix{\cdots & C^\infty_{r,\widetilde{W}_n}(M_n)^{K_n}\ar[d]_{\cF_n}&
C^\infty_{r,\widetilde{W}_{n+1}}(M_{n+1})^{K_{n+1}}
\ar[d]_{\cF_{n+1}}\ar[l]_(0.55){C_{n+1,n}}&\ar[l]_(0.2){C_{n+2,n+1}}
\quad\qquad \cdots &
C_{r, \infty}^\infty (M_\infty)^{K_\infty}\ar[d]_{\cF_\infty}
 \\
\cdots & \PW_r(\fa^*_{n,\C})^{\widetilde{W}_n}&  \PW_r(\fa^*_{n+1,\C})^{\widetilde{W}_{n+1}}
 \ar[l]^(0.55){P_{n+1,n}}
&\ar[l]^(0.2){P_{n+2,n+1}} \quad\qquad \cdots & \PW_r(\fa^*_{\infty, \C} )\\
}
$$

\begin{theorem}[Infinite dimensional Paley-Wiener Theorem]\label{th-ProjLimNonCompact}
Let the notation be as above. Then
$\PW_r(\fa^*_{\infty,\C})\not=\{0\}$, $C_{r,\infty}^\infty (M_\infty)^{K_\infty}\not=\{0\}$
and
\[\cF_\infty :C_{r,\infty}^\infty (M_\infty)^{K_\infty}
\to \PW_r(\fa^*_{\infty,\C})\]
is a linear isomorphism.
\end{theorem}

\section{Central Functions on Compact Lie Groups}\label{sec5}
\noindent
The following results on compact Lie groups are a special case of
the more general statements on compact symmetric spaces discussed
in the next section, as every group can be viewed as a symmetric
space $G\times G/\mathrm{diag}(G)$ via the map
\[(g,1)\mathrm{diag}(G)\mapsto g, \text{ in other words }
    (a,b)\mathrm{diag}(G)\mapsto ab^{-1}\]
corresponding to the involution $\tau (a,b)=(b,a)$. The action of
$G\times G$ is the left-right action $(L\times R)(a,b)\cdot x=
axb^{-1}$ and the $\textrm{diag}(G)$--invariant functions
are the \textit{central} functions $f(axa^{-1})= f(x)$ for
all $a,x\in G$.
Thus $f$ is central if and only if $f\circ \Ad (a)=f$ for
all $a\in G$, where as usual $\Ad (a)(x)=axa^{-1}$. From now on,
if $E$ is a function space on $G$, then $E^G$ denotes the space of
central functions in $E$.
But, because of the special role played by the group and the central
functions, it is worthwhile to discuss this case separately.

In this section $G$, $G_n$ and $G_k$ will
denote a compact connected semisimple Lie group. For simplicity, we will
assume that those groups are simply connected. The only change that need
to be made for the general case is to change the semi-lattice of
highest weights of irreducible representations and the injectivity
radius, whose numerical value does not play an important rule in the following.
We say that $G_k$ propagates $G_n$ if $\fg_k$ propagates $\fg_n$.
This is the same as saying that $G_k$ propagates $G_n$ as a symmetric space.
We fix a Cartan subalgebra $\fh_k$ of $\fg_k$ such that $\fh_n:=
\fh_k\cap \fg_n$ is a Cartan subalgebra of $\fg_n$. We use the notation
from the previous section. In the following we will introduce notations
for $G$. The index $n$ respectively $k$ will then denote the corresponding
object for $G_n$ respectively $G_k$. We fix an inner product
$\langle \cdot ,\cdot \rangle_n$ respectively $\langle \cdot ,\cdot \rangle_k$
on $\fg_n$ respectively $\fg_k$ such that
$\langle X,Y\rangle_n=\langle X,Y\rangle_k$ for $X,Y\in \fg_n\subseteqq \fg_k$.
This can be done by viewing $G_n\subset G_k$ as locally isomorphic to  linear groups
and use the trace form $X,Y\mapsto -\Tr(XY)$.
We denote by $R$ the injectivity radius; Theorem
\ref{inj-radius} below show that the injectivity radius is the same for
$G_n$ and $G_k$.

The following is a reformulation of results of Crittenden \cite{crit}.
\begin{theorem}\label{rjcrit}
The minimum locus and the first conjugate locus of $G$ coincide and are
given by $\Ad(G)\gf$ where
$\gf = \{X \in \fh \mid \max_{\alpha \in \Delta}|\alpha(X)| = 2\pi\}$,
and the injectivity radius $R = \min_{X \in \gf}\Vert X \Vert$.
\end{theorem}
\begin{remark}\label{also-symmetric}{\rm
Crittenden actually proves the analogous result for symmetric
spaces of compact type.  That will be used in Section \ref{sec7}.}
\hfill $\diamondsuit$
\end{remark}
\begin{proof} The ``roots'' of \cite{crit} are the ``angular parameters''
of Hopf and Stiefel: the formulae of \cite[Section 2]{crit} shows that
they are $2\pi i$ times what we now call roots.  Thus
\cite[Theorem 3]{crit} says that $X \in \fh$ belongs to the conjugate
locus just when there is a root $\alpha \in \Delta$ such that
$|\alpha(X)|$ is a nonzero multiple of $2\pi$.  So $X \in \fh$ belongs to
the first conjugate locus just when there is a root $\alpha \in \Delta$
such that $|\alpha(X)| = 2\pi$ but there is no root $\beta \in \Delta$
such that $|\beta(tX)| = 2\pi$ with $|t| < 1$.  In other words the
first conjugate locus is  $\Ad(G)\gf$ as asserted.

The statement of \cite[Theorem 5]{crit} is that the minimum locus is equal
to the first conjugate locus.  Now the injectivity radius $R$ is the minimal
length $\Vert X \Vert = \langle X,X\rangle^{1/2}$ of an element of the
first conjugate locus.  In other words $R = \min_{X \in \gf}\Vert X \Vert$.
\end{proof}

For $\alpha\in \Delta$ let $t_\alpha\in [\fg_\alpha,\fg_{-\alpha}]$
be so that $\alpha (t_\alpha)=2$.  Let $t_1,\ldots ,t_r$ be the
$t_{\alpha_i}$ for the simple roots $\alpha_1,\ldots ,\alpha_r$.
Let $\Gamma = \{H\in \fh \mid \exp H=e\}$.  Then
$\Gamma = \bigoplus_{j=1}^r  2\pi \Z  t_j$\,.
It follows that the injectivity radius is given by
\begin{equation}\label{re-inj-rad}
R=\min_{j=1,\ldots ,r} \pi \|t_j\|\, .
\end{equation}
Now we use (\ref{re-inj-rad}) to run through the four cases,
making Theorem \ref{rjcrit} explicit for our setting.

Here we use the matrix realization notation of (\ref{an}), (\ref{bn}),
(\ref{cn}) and (\ref{dn}), we use the realizations of roots as matrices
as introduced in Section \ref{sec2}, and we use the Riemannian metric on
$G$ defined
by the positive definite inner product $\langle X,Y\rangle = -\trace(XY)$.

$A_n$: We have $(f_i-f_j,f_i-f_j)=2$ hence $t_i=f_i -f_{i+1}$. It follows
that $R=\sqrt{2}\pi$.

$B_n$: The simple roots are the $f_1$ and the $f_i-f_j$.
Hence $t_1=2f_1$, and $t_i= f_i -f_{i+1}$ for $i > 1$.
The realization of $x \in \fh$ as a matrix is
\[x\mapsto \begin{pmatrix} 0 & & \\
& \diag (x) & \\
& & -\diag (x)\end{pmatrix}
\, .\]
Hence
$\|t\|=2\sqrt{2}\quad \text{and}\quad \|t_j\|=2$ and so
$R=2\pi$\,.

$C_n$: The simple roots are $2f_1$ and the $f_j-f_{j-1}$ for $j > 1$.
The realization of $x \in \fh$ as a matrix is
\[x\mapsto \begin{pmatrix} \diag (x) & \\
& -\diag (x) \end{pmatrix}\, .\]
That gives us $t_1=f_1$ and $t_j = f_j-f_{j-1}$ for $j > 1$.
Thus $\|t_1\| = \sqrt{2}$ and $\|t_j\|=2$ for $j > 1$, so
$R=\sqrt{2}\pi$\,.

$D_n$: The realization of $x \in \fh$ as a matrix is the same as for $C_n$.
The simple roots are $f_1+f_2=t_1$ and the $f_j-f_{j-1}=t_j$. Hence
\begin{theorem}\label{inj-radius}
The injectivity radius of the classical compact simply connected Lie groups
$G$, in the Riemannian metric given by the inner product
$\langle X,Y\rangle = -\trace(XY)$ on $\fg$, is $\sqrt{2}\,\pi$ for $SU(m+1)$
and $Sp(m)$, $2\pi$ for $SO(2m)$ and $SO(2m+1)$.  In particular
for each of the four series the injectivity radius $R$ is independent of $m$.
\end{theorem}

The invariant measures on $G$, $G_n$ and $G_k$ all are normalized to total
mass $1$.

We start by recalling Gonzalez' Paley-Wiener theorem \cite{Gon} (also
see \cite{OS}). Denote by  $\gL^+(G)\subset i\fh^*$ the set of dominant
integral weights,
\[\gL^+(G)=\left \{\mu\in i\fh^*)\left |
\,\, \tfrac{2 (\mu ,\alpha )}{(\alpha ,\alpha)}\in \Z^+ \text{ for all }
 \alpha \in \Delta^+(\fg_\C,\fh_\C)\right . \right \}\, .\]
For $\mu\in\gL^+(G)$ denote by
$\pi_\mu$ the corresponding representation with highest weight
$\mu$.  As $G$ is assumed
simply connected $\mu \mapsto \pi_\mu$, is a bijection from
$\gL^+(G)$ onto $\widehat{G}$. The representation
space for $\pi_\mu$ is denoted by $V_\mu$.  Let $\chi_\mu=\Tr \circ \pi_\mu$ be the character of $\pi_\mu$
and  $\deg(\mu )=\dim V_\mu$ its dimension. Note that
$\deg (\mu )$ extends to a polynomial function on $\fh_\C^*$.
As Haar measure is normalized to total mass $1$, the characters
$\{\chi_\mu\}_{\mu\in\gL^+(G)}$ form a complete orthonormal set for
$L^2(G)^G:= \{f \in L^2(G) \mid f\circ\Ad (g) = f \text{ for all } g \in G\}$.

For
$f\in C (G)^G$ define the Fourier transform $\cF(f)=\widehat{f}:\gL^+(G)\to \C$ by
$$\widehat f(\mu)= ( f,\chi_\mu )=\int_G f(x)\overline{\chi_\mu(x)}\,dx, \quad \mu\in\gL^+(G)\, ,$$
where $( f,\chi_\mu )$ is the inner product in $L^2(G)$.
The Fourier transform extends to an unitary isomorphism
$\cF : L^2(G)^G \to \ell^2(\gL^+(G))$ and
\[f=\sum_{\mu \in \gL^+(G)}\widehat{f}(\mu )\chi_\mu\]
in $L^2(G)^G$. If $f$ is smooth
the Fourier series converges absolutely and uniformly.

If not otherwise stated we will assume that $G$ does not contain
any simple factor of exceptional type.
As before $W(\fg,\fh)$ denotes the Weyl group of $\Delta (\fg_\C,\fh_\C)$,
$\widetilde{W}=\widetilde{W}(\fg,\fh)$
also denotes that Weyl group when $G$ is of type $A_n$, $B_n$ or $C_n$,
and $\widetilde{W}=\widetilde{W}(\fg,\fh)$ is the
Weyl group extended by including odd sign changes in the $D_n$ cases.
For $r>0$  let $\PW_r^\rho (\fh^*_\C)^{\widetilde{W}}$ denote the space
of holomorphic functions $\Phi$ on $\fh_\C^*$ such that
\begin{enumerate}
\item  For each $k\in\N$ there exists a constant $C_k>0$ such that
$$|\Phi(\lambda )|\leqq C_k(1+|\lambda |)^{-k} e^{r|\Re\lambda |}
\text{ for all } \lambda \in\fh_\C^*,
$$
\item $\Phi(w(\lambda +\rho)-\rho)=\det(w)\Phi(\lambda )$ for
all $w\in \widetilde{W}$, $\lambda \in\fh_\C^*$.
\end{enumerate}

Let $H=\exp (\fh)$.
For $0<r<R$ denote by $C_{r,\widetilde{W}}^\infty (G)^G$ the space of smooth central functions
with support in a closed geodesic ball $B_r(e)$ of radius $r$ such
that the restriction to $H$ is $\widetilde W$--invariant. We refer to a much more detailed
discussion in Section \ref{sec3}.
In this terminology the theorem of Gonzalez \cite{Gon} reads as follows.

\begin{theorem}\label{t: Gonzalez}
Let $G$ be an arbitrary connected simply connected compact
Lie group.
Let $0<r<R$ and let $f\in C^\infty(G)^G$ be given. Then $f$ belongs to
$C^\infty_{r,\widetilde{W}}(G)^G$  if and only if the
Fourier transform $\mu\mapsto \widehat{f}(\mu)$
extends to a holomorphic function $\Phi_f$
on $\fh^*_\C$ such that $\Phi_f\in \PW_r^\rho(\fh^*_\C)^{\widetilde{W}}$.
\end{theorem}

\begin{proof} We only have to check that $f\in C^\infty_{r,\widetilde{W}}(G)^G$
if and only if $\widehat{f}(w(\mu +\rho)-\rho )=
\widehat{f}(\mu )$. For factors not of type $D_n$ that follows from
Gonzalez's theorem. For factors of type $D_n$ it follows Weyl's character
formula.
\end{proof}
In \cite{OS} it is shown that the extension $\Phi_f$ is unique whenever
$r$ is sufficiently small. In that case Fourier transform,
followed by holomorphic extension, is a bijection
$C^\infty_{r,\widetilde{W}} (G)^G\cong \PW_r^\rho(\fh^*_\C )^{\widetilde{W}}$.

We will now extend these results to projective limits. We start with two
simple lemmas.

\begin{lemma}\label{le-PhZero}
Let $\Phi\in\PW_r^\rho(\fh^*_\C)^{\widetilde{W}}$.
Assume that $\lambda\in \fh^*_\C$ is
such that $\langle \lambda ,\alpha\rangle =0$ for some $\alpha\in\Delta$. Then
$\Phi (\lambda-\rho )=0$.
\end{lemma}
\begin{proof} Let $s_\alpha $ be the reflection in the hyper plane
perpendicular to $\alpha$. Then
\begin{eqnarray*}
\Phi (\lambda -\rho )&=&\Phi (s_\alpha (\lambda )-\rho )\\
&=& \Phi (s_\alpha (\lambda -\rho +\rho )-\rho) = \det (s_\alpha )\Phi (\lambda -\rho)\, .
\end{eqnarray*}
The claim now follows as $\det (s_\alpha )=-1$.
\end{proof}

\begin{lemma}\label{isoPWspaces}
Let $r>0$ and let $\widetilde{W}$ be as before.
For $\Phi \in \PW_r^\rho (\fh^*_\C)^{\widetilde{W}}$ define
\[
T(\Phi ) (\lambda )=F_\Phi (\lambda ):=
\tfrac{\varpi (\rho) }{\varpi (\lambda )}\Phi (\lambda -\rho )
\text{ where } \varpi (\lambda )=\prod_{\alpha\in\Delta^+}
    \langle\lambda ,\alpha \rangle\, .
\]
Then $T (\Phi )\in \PW_r (\fh^*_\C)^{\widetilde{W}}$ and the map
$\PW_r^\rho (\fh^*_\C)^{\widetilde{W}}\to \PW_r(\fh^*_\C)^{\widetilde{W}}$,
$\Phi \mapsto F_\Phi$, is a linear isomorphism.
\end{lemma}
\begin{proof} Let $\alpha \in\Delta^+$. Then
\[\lambda \mapsto \frac{1}{(\lambda ,\alpha)} \Phi (\lambda )\]
is holomorphic by Lemma \ref{le-PhZero}. According to
\cite{He1994}, Lemma 5. 13, p. 288, it follows that
this function is also of exponential type $r$. Iterating this for each root it follows that
$F_\Phi$ is holomorphic of exponential type $r$. As $\varpi (w(\lambda ))=\det (w)\varpi (\lambda )$
it follows using the same arguments as in the proof of Lemma \ref{le-PhZero} that
$F_\Phi$ is $\widetilde W$--invariant. The surjectivity follow as
$F\mapsto \varpi (\lambda )F(\cdot +\rho)$ maps $\PW_r(\fh^*_\C)^{\widetilde{W}}$
into $\PW_r^\rho
(\fh^*_\C)^{\widetilde{W}}$.
\end{proof}
\begin{theorem}\label{th-PknSurjective} Let $r>0$ and assume that
$G_k$ propagates $G_n$. Then the map
\[\Phi \mapsto P_{k,n}(\Phi):=T_n^{-1}( T_k(\Phi)|_{\fh_{n,\C}^*})=
\frac{\varpi_n(\bullet )}{\varpi_n(\rho_n)} \left(\frac{\varpi_k (\rho_k)}{\varpi_k (\bullet )}
\Phi (\bullet -\rho_k)|_{\fh_{n,\C}^*}
\right)(\bullet +\rho_n)  \]
from $\PW_r^{\rho_k}(\fh_{k,\C}^*)^{\widetilde{W}_k}\to
\PW_r^{\rho_n}(\fh_{n,\C}^*)^{\widetilde{W}_n}$ is surjective.
\end{theorem}
\begin{proof} This follows from Lemma \ref{isoPWspaces} and Theorem \ref{th-IsoPW1}.
\end{proof}

Recall from Theorem \ref{inj-radius} that the injectivity radii
$R$ are the same for $G_k$ and $G_n$.  For $0<r<R$
we now define a map $C_{k,n} :C^\infty_{r,\widetilde{W}_k}
(G_k)^{G_k}\to C^\infty_{r,\widetilde{W}_n}(G_n)^{G_n}$ by
the commutative diagram using Gonzalez' theorem:
$$\xymatrix{C^\infty_{r,\widetilde{W}_k}(G_k)^{G_k} \ar[d]_{\cF_k}\ar[r]^{C_{k,n}}&
C^\infty_{r,\widetilde{W}_n}(G_n)^{G_n}\ar[d]^{\cF_n}
 \\
\PW_r^{\rho_k}(\fh_{k,\C})^{\widetilde W_k}\ar[r]_{P_{k,n}} & \PW_r^{\rho_n}
(\fh_{n,\C})^{\widetilde W_n} \\
}\, .
$$
\begin{theorem}\label{th-CknSurjective} If $G_k$ propagates $G_n$ and $0<r<R$
then $$C_{k,n} : C^\infty_{r,\widetilde W_k}(G_k)^{G_k} \to
C^\infty_{r,\widetilde W_n}(G_n)^{G_n}$$ is surjective.
\end{theorem}
\begin{proof} This follows from Theorem \ref{t: Gonzalez} and
Theorem \ref{th-PknSurjective}.
\end{proof}

\begin{theorem} Let $r>0$ and assume that $G_k$ propagates $G_n$. Then
the sequences $(\PW_r^{\rho_n} (\fh_{n,\C}^*)^{\widetilde W_n},
P_{k,n})$ and $(C^\infty_{r,\widetilde W_n}(G_n)^{G_n}, C_{k,n})$
form projective systems and

\begin{eqnarray*}
\PW_{r}^{\rho_\infty}(\fh_{\infty,\C}) &:=&\varprojlim \, \PW_r^{\rho_n} (\fh_{n,\C}^*)^{\widetilde W_n}\\
C^\infty_{r,\infty}(G_\infty )^{G_\infty}&:=&\varprojlim \, C_{r,\widetilde W_n} ^\infty (G_n)^{G_n}
\end{eqnarray*}
are nonzero.
\end{theorem}
\begin{proof} This follows from Theorem \ref{th-PknSurjective} and
Theorem \ref{th-CknSurjective}.
\end{proof}
\begin{remark}\label{inf-holo}{\rm
We can view elements $\Phi \in \PW_{r}^{\rho_\infty}(\fh_{\infty,\C})$ as
holomorphic functions on $\fh_{\infty,\C}$ when we view $\fh_{\infty,\C}$ as
the spectrum of $\varprojlim  \PW_r^{\rho_n} (\fh_{n,\C}^*)$.}
\hfill $\diamondsuit$
\end{remark}

\section{Spherical Representations of Compact Groups}\label{sec6}
\setcounter{equation}{0}

\noindent
We will now apply the results from Section \ref{sec1} and Section \ref{sec2}
to the Fourier transform on compact symmetric spaces. We start by an
overview over spherical representations, spherical functions and
the spherical Fourier transform. Most of the material can be found in
\cite{W2008a} and \cite{W2008b} but
partially with different proofs. The notation will be as in Section \ref{sec3}
and $G$ or $G_n$ will always stand for a compact group.
In particular,
$M_n=G_n/K_n$ where $G_n$ is a connected compact semisimple Lie group
with Lie algebra $\fg_n$, which we will for simplicity
assume is simply connected. The result can easily be
formulated for arbitrary compact symmetric spaces by following the
arguments in \cite{OS}. We will assume that $M_k$ propagates $M_n$.
We denote by $r_k$ respectively $r_n$ the real rank
of $M_k$ respectively $M_n$. As always we fix compatible $K_k$-- and
$K_n$--invariant inner products on $\fs_k$ respectively $\fs_n$.

As in Section \ref{sec3}
let $\Sigma_n = \Sigma_n(\mfg_n , \ga_n)$ denote the
system of restricted roots of $\ga_{n,\C}$ in $\mfg_{n,\C}$.
Let $\fh_n$ be a $\theta_n$-stable Cartan subalgebra such
that $\fh_n\cap \fs_n =\fa_n$. Let
$\Delta_n=\Delta (\fg_{n,\C}, \fh_{n,\C})$. Recall
that $\Sigma_n\subset i\ga_n^*$.
We choose positive subsystems
$\Delta_n^+$ and $\Sigma^+_n $ so that
$\Sigma_n^+\subseteqq \Delta_n^+|_{\fa_n}$,
$\Delta_n^+\subseteqq \Delta_k^+|_{\fh_{n,\C}}$,  and
$\Sigma^+_n \subset \Sigma^+_k|_{\ga_n}$.
Consider the reduced root system
$$
\Sigma_{2,n}=\{\alpha\in\Sigma_n\mid 2\alpha\not\in\Sigma_n\}
$$
and its positive subsystem $\Sigma_{2,n}^+ := \Sigma_{2,n} \cap \Sigma^+_n$.
Let
$$
\Psi_{2,n} = \Psi_{2,n}(\mfg_n , \ga_n) = \{\alpha_{n,1}, \dots , \alpha_{n,r_n}\}
$$
denote the set of simple roots for $\Sigma_{2,n}^+$.
We note the following simple facts; they follow from
the explicit realization (\ref{rootorder}) of the root systems discussed
in Section \ref{sec2}.
\begin{lemma} \label{X} Suppose that the $M_n$ are irreducible.
Let $r_n=\dim \fa_n$, the rank of $M_n$.
Number the simple root systems $\Psi_{2,n}$ as in $(\ref{rootorder})$.
Suppose that $M_k$ propagates $M_n$.
If $j\leqq r_n$ then $\alpha_{k,j}$ is the unique element of
$\Psi_{2,k}$ whose restriction to $\fa_n$ is $\alpha_{n,j}$.
\end{lemma}

Since $M_k$ propagates $M_n$ each irreducible factor of $M_k$ contains
at most  one simple factor of $M_n$.  In particular if $M_n$ is not irreducible
then $M_k$ is not irreducible, but we still can number the simple
roots so that Lemma \ref{X} applies.

We denote the
positive Weyl chamber in $\fa_n$ by $\fa_n^+$ and similarly for
$\fa_k$.

For $\mu \in \gL^+(G_n)$ let
$$
V_{\mu}^{K_n}=\{v\in V_{\mu} \mid
\pi_{\mu}(k)v = v \text{ for all } k\in K_n\}.
$$
We identify $i\fa_n^*$ with $\{\mu \in i\fh_n^*\mid \mu|_{\fh_n\cap \fk_n}=0\}$ and
similar for $\fa_n^*$ and $\fa_{n,\C}^*$.
With this identification in mind set
$$
\Lambda^+(G_n,K_n)
= \left \{\mu\in i\fa_n^* \left |
\tfrac{ (\mu ,\alpha ) }{ ( \alpha ,\alpha )}\in \Z^+ \text{ for all }
\alpha \in \Sigma^+ \right . \right \}.
$$

Since $G_n$ is connected and $M_n$ is simply connected it follows that
$K_n$ is connected. As $K_n$ is compact there exists a unique (up to
multiplication by a positive scalar) $G_n$--invariant measure
$\mu_{M_n}$ on $M_n$.  For brevity we sometimes write $dx$ instead of
$d\mu_{M_n}$. If $G_n$ is compact, in other words if $M_n$ is compact,
then we normalize $\mu_{M_n}$ so that
$\mu_{M_n}(M_n)=1$, i.e., $\mu_{M_n}$ is a probability measure on $M_n$.

\begin{theorem}[Cartan-Helgason]\label{t-CH} Assume that $G_n$ is
compact and simply connected. Then the following are equivalent.
\begin{enumerate}
\item $\mu \in \Lambda^+(G_n,K_n)$,
\item $\displaystyle{V_{\mu}^{K_n} \ne 0}$,
\item $ \displaystyle{\pi_{\mu}}$ is a subrepresentation of the
representation of  $G_n$  on  $L^2(M_n)$.
\end{enumerate}
When those conditions hold, $\dim V_{\mu}^{K_n} = 1$ and
$\pi_{\mu}$ occurs with multiplicity $1$ in the representation of
$G_n$ on $L^2(G_n/K_n)$.
\end{theorem}
\begin{proof} See \cite[Theorem 4.1, p. 535]{He1984}.
\end{proof}

\begin{remark} {\rm If $G_n$  is compact but not simply connected, then one has
to replace $\Lambda_n^+$ and $\Lambda^+(G_n,K_n)$ by sub semi--lattices of
weights $\mu$ such that the group
homomorphism $\exp (X)\mapsto e^{\mu (X)}$ is well
defined on the maximal torus
$H_n$, and then the proof of Theorem \ref{t-CH} goes through without change.
See, for example, \cite{OS}.} \hfill $\diamondsuit$
\end{remark}

Define linear functionals $\xi_{n,j}\in i\ga_n^*$ by
\begin{equation}\label{fundclass1}
\frac{ \langle \xi_{n,i},\alpha_{n,j} \rangle }
{\langle \alpha_{n,j},\alpha_{n,j} \rangle} = \delta_{i,j} \text{ for }
1 \leqq j \leqq r_n\ \ .
\end{equation}
Then for $\alpha \in \Sigma_{2,n}^+$
$$\frac{ \langle  \xi_{n,i},\alpha \rangle  }
{ \langle  \alpha,\alpha \rangle  }\in \Z^+\, .$$
If $\alpha \in \Sigma^+\setminus \Sigma_{2,n}^+$, then
$2\alpha \in \Sigma_{2,n}^+$ and
$$\frac{ \langle  \xi_{n,i},\alpha \rangle }
{ \langle  \alpha,\alpha \rangle  }
= 2\frac{ \langle  \xi_{n,i},2\alpha \rangle  }
{ \langle 2\alpha,2 \alpha \rangle  }\in \Z^+\, .$$
Hence $\xi_{n,i}\in \Lambda^+_n$. The weights $\xi_{n,j}$ are the
\textit{class 1 fundamental weights for}
$(\mfg_n,\gk_n)$. We set
$$\Xi_n=\{\xi_{n,1},\ldots ,\xi_{n,r_n}\}\, .$$
For $I=(k_1,\ldots ,k_{r_n})\in (\Z^+)^{r_n}$ define
$\mu_I:=\mu(I)=k_1\xi_{n,1}+\ldots +k_{r_n}\xi_{n,r_n}$.
\begin{lemma}\label{le-KnWeights} If $\mu \in i\ga_n^*$ then
$\mu \in \Lambda^+ (G_n,K_n)$ if and only if
$\mu =\mu_I$ for some $I\in (\Z^+)^{r_n}$.
\end{lemma}
\begin{proof} This follows directly from the definition of
$\xi_{n,j}$.
\end{proof}
\begin{lemma}\label{le-ResInLambdan} Suppose that $M_k$ is a
propagation of $M_n$. Let
$I_k=(m_1,\ldots ,m_k)\in (\Z^+)^{r_k}$ and
$\mu =\mu_{I_k}$. Then
$\mu|_{\fa_n}\in \Lambda^+(G_n,K_n)$. In particular
$\xi_{k,j}|_{\fa_n}\in \Lambda^+ (G_n,K_n)$ for $1 \leqq j \leqq r_k$.
\end{lemma}
\begin{proof}
Let $v_\mu \in V_\mu$ be a nonzero highest weight vector and
$e_\mu \in V_\mu$ a $K_k$--fixed unit vector. Denote by
$W=\langle\pi_\mu (G_n)v_\mu \rangle$
the cyclic $G_n$-module generated by $v_\mu $ and let $\mu_n=\mu|_{\ga_n}$.

Write $W=\bigoplus_{j=1}^s W_j$ with $W_j$ irreducible.
If $W_j$ has highest weight $\nu_j \ne \mu$ then $v_\mu  \perp W_j$
so $\langle\pi_\mu (G_n)v_\mu \rangle \perp W_j$, contradicting
$W_j \subset W =\bigoplus W_i$.  Now each $W_j$ has highest weight
$\mu$.  Write $v_\mu =v_1+\ldots +v_s$ with $0 \ne v_j\in W_j$.
As $( v_\mu ,e_{\mu} ) \not= 0$ it follows that
$( v_j,e_{\mu}) \not= 0$ for some $j$. But then the projection
of $e_\mu$ onto $W_j$ is a non-zero $K_n$ fixed vector in $W_j^{K_n}\not= 0$
and hence $\mu |_{n}\in\Lambda^+(G_n,K_n)$.
\end{proof}

\begin{lemma}[\cite{W2008a}, Lemma 6]\label{simple-res}
Assume that $M_k$ is a propagation of $M_n$.  Recall the root ordering
of {\rm (\ref{rootorder})}.  If $1 \leqq j \leqq r_n$ then $\xi_{k,j}$
is the unique element of $\Xi_k$ whose restriction of $\ga_n$ is $\xi_{n,j}$.
\end{lemma}

\begin{proof} This is clear when $\fa_k=\fa_n$.
If $r_n<r_k$ it follows from the explicit construction
of the fundamental weights for classical root system; see
\cite[p. 102]{GW1998}.
\end{proof}

\begin{lemma}\label{resmult}
Assume that $\mu_k \in \Lambda^+(G_k,K_k)$ is a combination
of the first $r_n$ fundamental weights,  $\mu =
\sum_{j=1}^{r_n} k_j \xi_{k,j}$.
Let
$\mu_n:=\mu|_{\ga_n}=\sum_{j=1}^{r_n}k_j \xi_{n,j}\, $.
If $v$ is a nonzero highest weight vector in $V_{\mu_k}$ then
$\langle\pi_{\mu_k} (G_n)v\rangle$ is irreducible and isomorphic to $V_{\mu_n}$.
Furthermore, $\pi_{\mu_n}$
occurs with multiplicity one in $\pi_{\mu_k}|_{G_n}$.
\end{lemma}
\begin{proof} Each $G_n$--irreducible summand $W$
in $\langle\pi_{\mu_k} (G_n)v\rangle$ has highest weight $\mu_n$.
Fix one such $G_n$--submodule $W$ and let $w \in W$ be a nonzero highest
weight vector. Write $w=w_1+\ldots +w_s$ where each $w_j$ is of some
$\gh_k$--weight $\mu_k - \sum_ik_{j,i}\beta_i$ and where each
$\beta_i$ is a simple root in $\Sigma^+(\mfg_{k},\gh_{k})$
and each $k_{j,i}\in \Z^+$. As $\mu_k|_{\gh_n}=\mu_{n}$ it
follows that
$\langle \sum_{i}k_{j,i}\beta_i|_{\fh_n}, \alpha \rangle = 0$
for all $\alpha \in \Delta (\fg_n,\fh_n)$. Thus
$\sum_{i}k_{j,i}\beta_i|_{\fh_n}=0$. In view of (\ref{rootorder})
each $\langle \beta_i,\alpha_j\rangle \leqq 0$ for
$\alpha_j \in \Delta (\fg_n,\fh_n)$ simple (specifically
$\langle \beta_i,\alpha_j\rangle = 0$ unless $\beta_i = f_{c+1} - f_c$ and
$\alpha_j = f_c - f_{c-1}$, for some $c$, in which case
$\langle \beta_i,\alpha_j\rangle = -1$).  Since every $k_{j,i}\in \Z^+$
now $\langle \beta_i , \alpha_j \rangle = 0$ for each
$\alpha_j \in \Delta (\fg_n,\fh_n)$ simple.  Thus $\beta_i|_{\gh_n}=0$.

Because of the compatibility
of the positive systems $\Delta^+(\mfg_{k,\C},\gh_{k,\C})$ and
$\Delta^+(\mfg_{n,\C},\gh_{n,\C})$ there exists a $\beta \in
\Delta^+(\mfg_{k,\C},\gh_{k,\C})$, $\beta|_{\gh_n}=0$,  such that
$\mu_k -\beta$ is a weight in $V_{\mu_n}$.
Writing $\beta$ as a sum of simple roots, we see that each of the
simple roots
has to vanish on $\ga_n$ and hence the restriction to $\ga_{k}$
can not contain any of the simple roots $\alpha_{k,j}$, $ j=1,\ldots ,r_n$.
But then $\beta $ is perpendicular to the fundamental weights
$\xi_{k,j}$, $j=1,\ldots , r_n$.
Hence $s_\beta (\mu_n- \beta)=\mu_n +\beta$ is also a weight, contradicting
the fact that $\mu_n$ is the highest weight. (Here $s_\beta$ is the reflection
in the hyperplane $\beta =0$.) This shows that $\pi_{\mu_n}$ can
only occur once in $\langle \pi_{\mu_k }(G_n)v\rangle$. In particular,
$\langle \pi_{\mu_k}(G_n)v\rangle$ is irreducible.
\end{proof}

Lemma \ref{resmult} allows us to form direct system of representations, as
follows.  For $\ell \in \N$ denote by $0_\ell = (0,\ldots ,0)$ the zero vector
in $\R^\ell$.  For
$I_{n}=(k_1,\ldots ,k_{r_n})\in (\Z^+)^{r_n}$ let
\begin{equation}\label{I-notation}
\begin{aligned}
\bullet\ &\mu_{I,n}= {\sum}_{j=1}^{r_n}k_j\xi_{n,j}\in \Lambda^+_n;\\
\bullet\ &\pi_{I,n} = \pi_{\mu_{I,n}} \text{ the corresponding spherical
representation};\\
\bullet\ &V_{I,n} = V_{\mu_{I,n}}  \text{ a fixed Hilbert space for the
representation } \pi_{I,n};\\
\bullet\ &v_{I,n} = v_{\mu_{I,n}}  \text{ a highest weight unit vector in }
V_{ I,n};\\
\bullet\ &e_{I,n} = e_{\mu_{I,n}} \text{ a } K_n\text{--fixed unit vector in }
V_{I,n}.
\end{aligned}
\end{equation}

We collect our results in the following Theorem.  Compare
\cite[Section 3]{W2008a}.

\begin{theorem}\label{l-inductiveSystemOfRep}
Let $M_k$ propagate $M_n$ and
let $\pi_{I,n}$ be an irreducible representation of $G_n$
with highest weight $\mu_{I,n}\in\Lambda^+(G_n,K_n)$.
Let $I_k=(I_n,0_{r_k-r_n})$. Then the following hold.
\begin{enumerate}
\item $\mu_{I,k}\in \Lambda^+ (G_k,K_k)$ and
$\mu_{I,k}|_{\ga_n}=\mu_{I,n}$.
\item The $G_n$-submodule of $V_{I,k}$ generated
by $v_{I,k}$ is irreducible.
\item The multiplicity of $\pi_{I,n}$ in
$\pi_{I,k}|_{G_n}$ is $1$, in other words there is an unique
$G_n$--intertwining operator $T_{k,n}: V_{I,n}\to V_{I,k}$ such that
$$T_{k,n}(\pi (g)v_{I,n})=\pi_{I,k}(g)v_{I,k}\, .$$
\end{enumerate}
\end{theorem}

\begin{remark} {\rm From this point on, when $m\leqq q$ we will always
assume that the Hilbert space $V_{I,m}$ is realized inside $V_{I,q}$ as
$\langle \pi_{I,q}(G_m)v_{I,q}\rangle$.}\hfill $\diamondsuit$
\end{remark}

\section{Spherical Fourier Analysis and the Paley-Wiener Theorem}\label{sec7}

\noindent
In this section we give a short description of
the spherical functions and Fourier analysis on compact
symmetric spaces. Then we state and prove results for
limits of compact symmetric spaces analogous to those
in Section \ref{sec4}.

For the moment let $M=G/K$ be a compact symmetric space. We use
the same notation as in the last section but without the index
$n$. As usual we view functions on $M$ as right $K$--invariant
functions on $G$ via $f(g)=f(g\cdot x_o)$, $x_o=eK$.
For $\mu \in \Lambda (G,K)$ denote by $\deg (\mu )$ the dimension
of the irreducible representation $\pi_\mu$. Fix a unit $K$-fixed
vector $e_\mu$ and define
\[\psi_\mu (g)=( e_\mu ,\pi_\mu (g)e_\mu) \, .\]
Then $\psi_\mu$ is positive definite spherical function
on $G$, and every positive definite spherical function is
obtained in this way for a suitable representation $\pi$. Define
\begin{equation}\label{def-ell}
\ell^2_d(\Lambda^+ (G,K))=
\left \{\{a_\mu \}_{\mu\in\Lambda^+ (G,K)}\left |
a_\mu\in\C\,\,\mathrm{and}\right . \,\,
\sum_{\mu\in\Lambda^+(G,K)}\deg(\mu )|a_\mu |^2<\infty\right \}\, .
\end{equation}
Then $\ell^2_d(\Lambda^+ (G,K))$ is a Hilbert space with
inner product
\[((a(\mu))_\mu ,(b(\mu ))_\mu)=\sum_{\mu\in \Lambda^+(G,K)} \deg(\mu) a(\mu )\overline{b(\mu )}\, .\]
For $f\in C^\infty (M)$ define the spherical Fourier transform of $f$,
$\cS (f)=\widehat{f} :\Lambda^+(G,K)\to \C$ by
\[\widehat{f}(\mu )=(f,\psi_\mu)=\int_M f(g )(\pi_\mu (g)e_\mu ,e_\mu)\, dg
=(\pi_\mu (f)e_\mu ,e_\mu)\]
where $\pi_\mu (f)$ denotes the operator valued Fourier transform of $f$,
$\pi_\mu (f)=\int_G f(g) \pi_\mu (g)\, dg$.
Then the sequence $\cS(f)=(\cS(f)(\mu ))_\mu$ is in $\ell^2_d(\Lambda^+(G,K))$ and
$\|f\|^2=\|\cS (f)\|^2$. Finally, $\cS$ extends by continuity
to an unitary isomorphism
\[\cS : L^2(M)^K\to \ell^2_d(\Lambda^+(G,K))\, .\]
We denote by $\cS_\rho$ the map
\begin{equation}\label{def-Srho}
\cS_\rho (f)(\mu )=\cS(f)(\mu -\rho )\, ,\quad \mu\in \Lambda^+(G,K)+\rho\, .
\end{equation}
If $f$ is smooth, then $f$ is given by
\[f(x)=\sum_{\mu\in\Lambda^+(G,K)}\deg(\mu )\cS (f) (\mu )\psi_\mu (x)=
\sum_{\mu\in\Lambda^+(G,K)} \deg(\mu )\cS_\rho (f)(\mu +\rho )\psi_{\mu}(x)\, .\]
and the series converges in the usual Fr\'echet topology on $C^\infty (M)^K$.
In general, the sum has to be interpreted as an $L^2$ limit.

Let
\[
\Omega:=\{X\in \fa\mid
|\alpha (X)|<\pi/2 \text{ for all } \alpha\in\Sigma\}\, .\]
For $\lambda \in\fa_\C^*$ let $\varphi_\lambda$ denote the
spherical function on the dual symmetric space of noncompact
type $G^d/K$, where the Lie algebra of $G^d$ is given
by $\fk+i\fs$. Then $\varphi_\lambda$ has a
holomorphic extension as $K_\C$--invariant function
to $K_\C\exp (2\Omega )\cdot x_o\subset
G_\C/K_\C$, cf. \cite[Theorem 3.15]{Opd},
\cite{BOP2005} and \cite{KS2005}. Furthermore
\[\overline{\psi_\mu (x)}=\varphi_{\mu +\rho}(x^{-1})\]
for $x\in K_\C\exp (2\Omega )\cdot x_o$.
We can therefore define a holomorphic function $\lambda \mapsto \cS_\rho (f)(\lambda )$
by
\[\cS_\rho (f)(\lambda )=\int_M f(x)\varphi_{\lambda }(x^{-1})\, dx\]
as long as $f$ has support in $K_\C \exp (2\Omega)\cdot x_o$.
$\cS_\rho (f)$ is $W(\fg,\fa )$ invariant and
$\cS_\rho (f)(\mu )=\cS (f)(\mu -\rho )$ for all $\mu \in \Lambda^+(G,K)+\rho$.

Denote by $R$ the injectivity radius of the Riemannian exponential map
$\Exp :\fs \to M$.  As noted in Remark \ref{also-symmetric},
Theorem \ref{rjcrit} holds for compact simply connected Riemannian
symmetric spaces \cite{crit} generally, leading to the following extension of
Theorem \ref{inj-radius}.
\begin{theorem}\label{also-inj-radius}
The injectivity radius $R$ of the classical compact simply connected
Riemannian symmetric spaces $M = G/K$, in the Riemannian metric given by the
inner product $\langle X,Y\rangle = -\trace(XY)$ on $\fs$, depends
only on the type of the restricted reduced root system
$\Sigma_2(\fg_\C,\ga_\C)$.
It is $\sqrt{2}\,\pi$ for $\Sigma_2(\fg_\C,\ga_\C)$ of type $A$ or $C$ and is
$2\pi$ for $\Sigma_2(\fg_\C,\ga_\C)$ of type $B$ or $D$.
\end{theorem}
\begin{remark}{\em
Since $\Omega$ is given by $|\alpha(X)| < \pi/2$ and the interior of the
injectivity radius disk is given by $|\alpha(X)| < 2\pi$ the set $\Omega$
is contained in the open disk in $\gs$ of center $0$ and radius $R/4$.}
\hfill $\diamondsuit$
\end{remark}

Essentially as before, $\Br$ denotes the closed
metric ball in $M$ with center $x_o$ and radius $r$, and
$C^\infty_r(M)^K$ denotes
the space of $K$--invariant smooth functions on $M$ supported in $\Br$.

\begin{remark}\label{explain}{\rm
Theorem \ref{t: PW} below is, modulo a $\rho$-shift and $W$-invariance,
Theorem 4.2 and Remark 4.3 of \cite{OS}. As pointed
out in \cite[Remark 4.3]{OS}, the known value for the
constant $S$ can be different in each part of the theorem.
In Theorem\ref{t: PW}(1) we need that $S<R$ and the closed ball in $\fs$ with
center zero and radius $S$ has to be contained in
$K_\C\exp (i\Omega )\cdot x_o$
to be able to use the estimates from \cite{Opd} for the spherical
functions to show that we actually end up in the Paley-Wiener space.

In Theorem \ref{t: PW}(2) we need only that $S<R$. Thus the constant in
(1) is smaller than the one in (2). That is used in part (3).
For Theorem \ref{t: PW}(4) we also need $\|X\|\leqq \pi/\|\xi_j\|$ for
$j=1,\ldots ,r$.} \hfill $\diamondsuit$
\end{remark}

The group $\widetilde{W}=\widetilde{W}(\fg,\fh)$
is defined as before
and $C_{r,\widetilde{W}}^\infty(M)^K$ denotes the space of smooth $K$--invariant functions
with support in $B_r$ such that $f|_A$ is $\widetilde{W}$--invariant.

\begin{theorem}[Paley-Wiener Theorem for Compact Symmetric Spaces]
\label{t: PW} Let the notation be as above. Then the following
hold.
\begin{enumerate}
\item[1.] There exists a constant $S>0$ such that,
for each $0<r<S$ and  $f\in C^\infty_{r, \widetilde W}(M)^{K}$, the
$\rho$-shifted spherical  Fourier transform $\cS_\rho (f) : \Lambda^+_n+\rho \to\C$
extends to a function in $\PW_r(\ga_\C)^{\widetilde W}$.
\item[2.]There exists a constant $S>0$ such that if $F\in\PW_r(\ga_\C)^{\widetilde W}$, $0<r<S$,
the function
\begin{equation}\label{defOfIrho}
f(x):=\sum_{\mu\in\Lambda^+ (G,K)} \deg (\mu ) F (\mu + \rho )\psi_{\mu} (x)
\end{equation}
is in $C^\infty_{r,\widetilde W}(M)^K$ and
$\cS_\rho {f}(\mu )=F (\mu )$.
\item[3.] For $S$ as in $(1.)$ define $\cI_\rho :
\PW_r(\ga_\C)^{\widetilde W}\to C^\infty_{r,\widetilde W}(M)^K$ by {\rm (\ref{defOfIrho})}.
Then  $\cI_\rho$ is surjective for all $0<r<S$.
\item[4.] There exists a constant $S>0$ such that
for all $0<r<S$ the map $\cS_\rho$ followed by holomorphic extension
defines a bijection $C_{r,\widetilde {W}}(M)^{K}\cong \PW_r(\ga_\C )^{\widetilde W}$.
\end{enumerate} \end{theorem}

\begin{proof} As mentioned above this is Theorem 4.2 and Remark 4.3 in
\cite{OS} except for the $\widetilde W$-invariance. But that has only be checked
for factors of type $D_n$, where it follows as in the proof of Theorem
\ref{t: Gonzalez} by Weyl's character formula.
\end{proof}

A weaker version of the following theorem was used in  \cite[Section 11]{OS}.
It used an operator $Q$ which we will define shortly, and
some differentiation,
to prove the surjectivity part of local Paley--Wiener Theorem.
Denote the Fourier transform of $f\in C(G)^G$ by
$\cF (f)$. Recall the
operator $T:\PW_r^\rho(\fh_\C^*)^{\widetilde{W}(\fg,\fh)}\to
\PW_r(\fh_\C^*)^{\widetilde{W}(\fg,\fh)}$ from Theorem \ref{isoPWspaces}. Finally,
for $f\in C (G)$ let $f^\vee (x)=f(x^{-1})$. Then
${}^\vee : C_{r,\widetilde{W}(\fg,\fh)}^\infty (G)^G\to  C_{r,\widetilde{W}(\fg,\fh)}^\infty (G)^G$
is a bijection. We will identify $\fa_\C^*$ with the subspace
$\{\lambda\in\fh_\C^*\mid \lambda|_{\fh_\C\cap \fk_\C}=0\}$ without comment
in the following.

\begin{theorem} \label{stronger}
Let $S>0$ be as in {\rm Theorem \ref{t: PW}(1)}
and let $0<r<S$. Then the the
restriction map $\PW_r(\fh^*_\C)^{\widetilde{W}(\fg,\fh)}\to
\PW_r(\fa_\C^*)^{\widetilde{W}(\fg,\fa)}$ is surjective. Furthermore, the map
$C_{r,\widetilde W (\fg,\fh)}^\infty (G )^G\to
C^\infty_{r,\widetilde W (\fg,\fa)}(M)^K$, given by
\[Q(\varphi )(g\cdot x_o )=\int_K \varphi (gk)\, dk,\]
is surjective, and $\cS_\rho \circ Q(f^\vee )= T\circ \cF (f)$ on
$\Lambda^+(G,K)+\rho$.
\end{theorem}

\begin{proof} Surjectivity of the restriction map follows from
Theorem \ref{th-IsoPW1} and Theorem \ref{th-IhIa}.
Next, we have $Q (\chi_\mu^\vee )(x)=\int_K \chi_\mu (x^{-1}k)\, dk$.
As $\int_K \pi_\mu (k)\, dk$ is the
orthogonal projection onto $V_\mu^K$ it follows that $Q (\chi_\mu^\vee)=0$ if $\mu\not\in \Lambda^+(G,K)$ and
\[Q(\chi_\mu^\vee) (x)  =
(\pi_\mu (x^{-1})e_\mu,e_\mu)=(e_\mu ,\pi_\mu (x)e_\mu )=\psi_\mu (x)\]
for $\mu \in \Lambda^+(G,K)$. Thus, if
$f =\sum_{\mu} \cF (f)(\mu )\chi_\mu $ we have
$$
Q(f^\vee ) (x)= \sum_{\mu\in \Lambda^+(G,K)}\cF (f)(\mu )\psi_{\mu }(x)
=\sum_{\mu\in \Lambda^+(G,K)}\deg(\mu)\frac{\cF (f)(\mu )}{\deg(\mu)} \psi_{\mu }(x).
$$
Using the Weyl dimension formula for finite dimensional representations,
$\deg(\mu )=\frac{\varpi (\mu +\rho)}{\varpi (\rho )}$, we get
\[\cS_\rho (Q(f^\vee ))(\mu +\rho)=
\tfrac{\varpi (\mu +\rho)}{\varpi (\rho )}\, \cF (f)(\mu  )
=T(\cF (f))|_\fa (\mu +\rho)\]
for $\mu \in\Lambda^+(G,K)$. Hence $\cS_\rho\circ Q(f^\vee)|_{\Lambda^+(G,K)}
= (T\circ \cF (f)|_{\fa_\C})|_{\Lambda^+(G,K)}$.

Assume that $f\in C^\infty_{r,\widetilde{W}(\fg,\fa)}(G/K)^K$. Then, by
the Paley-Wiener Theorem, Theorem \ref{t: PW}, there
exists a $\Phi\in\PW_r(\fa_\C^*)^{\widetilde{W}(\fg,\fa)}$ such that
$\Phi =\cS_\rho (f)$ on $\Lambda^+(G,K)$.
Then, by what we just proved, there exists $\Psi\in \PW_r(\fh_\C^*)^{\widetilde{W}(\fg,\fh)}$
such that $\Psi|_{\fa_\C}=\Phi$.
By  Theorem \ref{t: Gonzalez} there
exists $F\in C_{r,\widetilde{W}(\fg,\fh)}(G)^G$ such that
$T\circ \cF (F)= \Psi$. By the above calculation we have
\[\cS (f) (\mu) =\cS (Q(F^\vee )) (\mu )\quad \text{for all}\quad \mu\in\Lambda^+(G,K)\, .\]
As clearly $Q (F^\vee)$ is smooth, it follows that $Q(F^\vee )=f$ and
hence $Q$ is surjective.
\end{proof}

Let $\sigma = 2(\alpha_1+\ldots +\alpha_\ell)$ where the
$\alpha_j\in \Sigma_2^+(\fg_\C,\fa_\C)$ are the simple roots.
For $M$ irreducible let
\begin{equation}\label{eq-Omega1+2}
\begin{aligned}
\Omega^*:=\,\, &\Omega \text{ if $\Sigma_2(\fg_\C,\fa_\C)$ is of type
$A_\ell$ or $C_\ell$}, \\
\Omega^*:= &\bigcap_{w\in W(\fg,\fa )}\{X\in\fa\mid |\sigma (w(X))| < \pi/2\}
\text{ if $\Sigma_2(\fg_\C,\fa_\C)$ is of type $B_\ell$ or $D_\ell$}.
\end{aligned}
\end{equation}
In general, we define $\Omega^*$ to be the product of the $\Omega^*$'s for all
the irreducible factors. Then $\Omega^*$ is a convex Weyl group invariant
polygon in $\fa$.  We also have $\Omega^* = -\Omega^*$. This is easy to
check and in any case will follow from our explicit description of $\Omega^*$.

Using the explicit realization of the irreducible root systems in
Section \ref{sec2} we describe the domains $\Omega^*$ in the following way:

\noindent
$\mathbf{A_n}$\textbf{:} We have $\fa =\{x\in\R^{n+1}\mid \sum x_j=0\}$,
$n\geqq 1$, and the roots are the $f_i-f_j: x\mapsto x_i-x_j$ for
$i\not= j$. Hence
\begin{equation}\label{eq-OmegaAn}\Omega^*=
    \Omega=\left \{x\in \R^{n+1}\left | \sum x_j=0\,\,\right .
\text{ and } |x_i-x_j| < \tfrac{\pi}{2} \text{ for }
1\leqq i\not= j\leqq n+1\right \}\, .
\end{equation}

\noindent
$\mathbf{B_n}$\textbf{:} We have $\fa=\R^n$, $n\geqq 2$ and $\sigma =
2(f_1+(f_2-f_1)+\ldots + (f_n-f_{n-1}))=2f_n$. The Weyl
group consists of all permutations and sign changes on the $f_i$. Hence
\begin{equation}\label{eq-OmegaBn}
\Omega^*=\{x\in\R^n\mid  |x_j| < \tfrac{\pi}{4} \text{ for } j=1,\ldots ,n\}\, .
\end{equation}

\noindent
$\mathbf{C_n}$\textbf{:} Again $\fa=\R^n$, $n\geqq 3$, and the roots are the
$\pm (f_i\pm f_j)$ and $\pm 2f_j$. If $|x_i|,|x_j|<\pi /4$ then
$|x_i\pm x_j|<\pi /2$. Hence
\begin{equation}\label{eq-OmegaCn}\Omega^*=\Omega=
\{x\in\R^n\mid |x_j|<\tfrac{\pi}{4} \text{ for } j=1,\ldots ,n\}\, .
\end{equation}

\noindent
$\mathbf{D_n}$\textbf{:} Also in this case $\fa=\R^n$ with $n\geqq 4$. We have
$\sigma = 2(f_1+f_2 + (f_2-f_1)+\ldots + (f_n-f_{n-1}))=2(f_2+f_n)$. As the
Weyl group is given by all permutations and even sign changes on the $f_i$,
we get
\begin{equation}\label{eq-OmegaDn}\Omega^*=
\{x\in\R^n \mid |x_i\pm x_j| < \tfrac{\pi}{4} \text{ for }
i,j=1,\ldots ,n\, ,\,\, i\not= j\}.
\end{equation}

\begin{lemma} We have $\Omega^*\subseteqq \Omega$.
\end{lemma}
\begin{proof}

Let $\delta$ be the highest root in $\Sigma^+$. Then
\[\Omega =W(\fa)(\{X\in \overline{\fa^+} \mid \delta (X) < \pi/2\})\, .\]
For the classical Lie algebras, the coefficients
of the simple roots in the highest root are all $1$ or $2$.  Hence
$\Omega^*\subseteqq \Omega$ and the claim follows.
\end{proof}

\begin{remark} \label{leftend} {\rm  The distinction between $\Omega$
and $\Omega^*$ is caused by change in the coefficient in the highest
root of the simple root on the left.  Thus in cases $B_n $ and $D_n$
it goes from $1$ to $2$ as we move up in the rank of $M$:
\begin{eqnarray*}
B_\ell &:&
\setlength{\unitlength}{.5 mm}
\begin{picture}(100,10)
\put(3,2){\circle{2}}
\put(0,5){$1$}
\put(4,2){\line(1,0){23}}
\put(28,2){\circle{2}}
\put(25,5){$2$}
\put(29,2){\line(1,0){13}}
\put(48,2){\circle*{1}}
\put(51,2){\circle*{1}}
\put(54,2){\circle*{1}}
\put(59,2){\line(1,0){13}}
\put(73,2){\circle{2}}
\put(70,5){$2$}
\put(74,2.5){\line(1,0){23}}
\put(74,1.5){\line(1,0){23}}
\put(98,2){\circle*{2}}
\put(95,5){$2$}
\end{picture}
\\
\\
D_\ell &:&
\setlength{\unitlength}{.5 mm}
\begin{picture}(100,10)
\put(3,2){\circle{2}}
\put(0,5){$1$}
\put(4,2){\line(1,0){23}}
\put(28,2){\circle{2}}
\put(25,5){$2$}
\put(29,2){\line(1,0){13}}
\put(48,2){\circle*{1}}
\put(51,2){\circle*{1}}
\put(54,2){\circle*{1}}
\put(59,2){\line(1,0){13}}
\put(73,2){\circle{2}}
\put(68,5){$2$}
\put(74,1.5){\line(2,-1){13}}
\put(88,-5){\circle{2}}
\put(91,-7){$1$}
\put(74,2.5){\line(2,1){13}}
\put(88,9){\circle{2}}
\put(91,7){$1$}
\end{picture}
\end{eqnarray*}
while in cases $A_n$ and $C_n$ it doesn't change:
\begin{eqnarray*}
A_\ell&:&
\setlength{\unitlength}{.5 mm}
\begin{picture}(100,15)
\put(3,2){\circle{2}}
\put(0,5){$1$}
\put(4,2){\line(1,0){23}}
\put(28,2){\circle{2}}
\put(25,5){$1$}
\put(29,2){\line(1,0){23}}
\put(53,2){\circle{2}}
\put(50,5){$1$}
\put(54,2){\line(1,0){13}}
\put(72,2){\circle*{1}}
\put(75,2){\circle*{1}}
\put(78,2){\circle*{1}}
\put(84,2){\line(1,0){13}}
\put(98,2){\circle{2}}
\put(95,5){$1$}
\end{picture}\\
\\
C_\ell &:&
\setlength{\unitlength}{.5 mm}
\begin{picture}(100,10)
\put(3,2){\circle*{2}}
\put(0,5){$2$}
\put(5,2){\line(1,0){23}}
\put(28,2){\circle*{2}}
\put(28,5){$2$}
\put(29,2){\line(1,0){13}}
\put(48,2){\circle*{1}}
\put(51,2){\circle*{1}}
\put(54,2){\circle*{1}}
\put(59,2){\line(1,0){13}}
\put(73,2){\circle*{2}}
\put(70,5){$2$}
\put(74,2.5){\line(1,0){23}}
\put(74,1.5){\line(1,0){23}}
\put(98,2){\circle{2}}
\put(95,5){$1$}
\end{picture}
\end{eqnarray*}
}\hfill $\diamondsuit$
\end{remark}

\begin{lemma}\label{le-Omega*} If $S>0$ such that
$\{X\in \fs \mid \|X\|\leqq S\}\subset \Ad (K) \Omega^*\}$, then
we can use $S$ as the constant in {\rm Theorem \ref{t: PW}(1)}.
\end{lemma}
\begin{proof} Recall from \cite[Remark 4.3]{OS} that
Theorem \ref{t: PW}(1) holds when $0 < S < R$ and
\begin{equation}\label{eq-Omega}
\{X\in \fs \mid \|X\|\leqq S\}\subseteqq \Ad (K) \Omega\, .
\end{equation}
But $\Ad (K)\Omega$ is open in $\fs$, and $\Exp : \Ad (K)\Omega \to M$ is
injective by Theorem \ref{also-inj-radius}.
Hence, if (\ref{eq-Omega}) holds then $S<R$, and the claim follows from
the first part of Remark \ref{explain}.
\end{proof}

We will now apply this to sequences $\{M_n\}$ where $M_k$ is a propagation of
$M_n$ for $k \geqq n$. We use the same notation as before and add the index
$n$ (or $k$) to indicate the dependence of the space $M_n$ (or $M_k$).
We start with the following lemma.
\begin{lemma}\label{le-Omega2*}  If $k\geqq n$ then
$\Omega_n^*=\Omega^*_k\cap \fa_n$.
\end{lemma}
\begin{proof} We can assume that $M$ is irreducible. As $M_k$ propagates
$M_n$ it follows that we are only adding simple roots to the left
on the Dynkin diagram for $\Sigma_2$. Let $r_n$ denote the rank of $M_n$ and
$r_k$ the rank of $M_k$. We can assume that $r_n < r_k$, as the claim is obvious
for $r_n = r_k$. We use the above explicit description $\Omega^*$ given above
and case by case
inspection of each of the irreducible root system in Section \ref{sec2}.

Assume that $\Sigma_{n,2}$ is of type $A_{r_n}$ and $\Sigma_{k,2}$ is
of type $A_{r_k}$ with $r_n < r_k$. It follows from (\ref{eq-OmegaAn}) that
$\Omega_n^*\subseteqq \Omega_k^*\cap \fa_n$.
Let $(0,x)\in \Omega_n^*$.
For $j>i$ we have
\begin{equation}\label{eq-alpha(0,x)}
\pm (f_{j}-f_i) ((0,x))=\left\{\begin{array}{c@{\quad\text{for}\quad}l}
\pm (x_j-x_i)&j\leqq r_n+1\\
\mp (-x_i)&j>r_n+1\geqq i\\
0&j,i>r_n+1
\end{array}\right.
\end{equation}
Let $i\leqq r_n+1$. Then, using that $x_i=-\sum_{j\not=i}x_j$ and
$|x_i-x_j| <\pi/2$, we get
$$
-r_k\tfrac{\pi}{2} <   \sum_{i\ne j} (x_i - x_j) = r_k x_i-\sum_{j\not= i}x_j
=(r_k+1)x_i < r_k\tfrac{\pi}{2}\, .
$$
Hence
\[-\tfrac{\pi}{2}<-\tfrac{r_k}{r_k+1}\, \tfrac{\pi}{2}<x_i<\tfrac{r_k}{r_k+1}\, \tfrac{\pi}{2}<\tfrac{\pi}{2}\, .\]
It follows now from (\ref{eq-alpha(0,x)}) that $(0,x)\in \Omega_k^*\cap \fa_n$.

The cases of types $B$ and $C$ are obvious from (\ref{eq-OmegaBn}) and
(\ref{eq-OmegaCn}). For the case of type $D$
we note that $|x_i\pm x_j|< \tfrac{\pi}{4}$ implies both
-$\tfrac{\pi}{4} < x_i - x_j < \tfrac{\pi}{4}$ and
-$\tfrac{\pi}{4} < x_i + x_j < \tfrac{\pi}{4}$.  Adding,
-$\tfrac{\pi}{2} < 2x_i < \tfrac{\pi}{2}$, so $|x_i| < \tfrac{\pi}{4}$.
Hence $(0,x)\in \Omega_k^*\cap \fa_n$ if and only if
$x\in \Omega_n^*$ by (\ref{eq-OmegaDn}).
\end{proof}

We can now proceed as in Section \ref{sec4}. We will always assume that
$S>0$ is so that the closed ball in $\fs$ of radius $S$ is contained in
$\Omega^*$. The group $W$ is defined as before. Define
$C_{k,n}:C^\infty_{r,W_k}(M_k)^{K_k}\to C^\infty_{r,W_n}(M_n)^{K_n}$ by
$C_{k,n}:=\cI_{n,\rho_n} \circ P_{k,n}\circ \cS_{k, \rho_k}$, in other words
\[C_{k,n}(f)(x)=\sum_{I\in(\Z^+)^{r_n}} \deg(\mu_{I,n})\widehat{f}
(\mu_{I,k}-\rho_k+\rho_n) \psi_{\mu_{I,n}}(x)\, .\]

\begin{theorem}[Paley-Wiener Isomorphisms-II]\label{th-PWcompactII} Assume that
$M_k$ propagates $M_n$ and $0<r<S$.  Then
the following holds:
\begin{enumerate}
\item The map $P_{k,n}:
\PW_r (\fa_{k,\C}^*)^{W_k} \to \PW_r(\fa_{n,\C}^*)^{W_n}$ is surjective.
\item The map $C_{k,n} :C^\infty_{r,W_k}(M_k)^{K_k}
\to C^\infty_{r,W_n}(M_n)^{K_n}$ is surjective.
\end{enumerate}
\end{theorem}
\begin{proof} This follows from  Theorem \ref{th-IsoPW1}, Lemma \ref{le-Omega*},
and Lemma \ref{le-Omega2*}.
\end{proof}

We assume now that $\{M_n,\iota_{k,n}\}$ is a injective system of
Riemannian symmetric spaces of compact type
such that $M_k$ is a propagation of $M_n$ along a cofinite subsequence.
Here the direct system maps $\iota_{k,n}: M_n\to M_k$ are injections.
We pass to that cofinite subsequence and now assume that
$M_k$ is a propagation of $M_n$ whenever $k \geqq n$.  Let
\[M_\infty =\varinjlim M_n\, .\]

The compact symmetric spaces in Table \ref{symmetric-case-class} give
rise to the following injective limits of symmetric spaces.

\begin{equation}\label{e-infiniteDim}
\begin{aligned}
{\rm 1.}\ &\bigl (\SU (\infty) \times \SU(\infty)\bigr )/\diag\,\SU(\infty),
    \text{ group manifold } \SU(\infty),\\
{\rm 2.}\ &\bigl (\Spin(\infty)\times \Spin(\infty)\bigr )/\diag\, \Spin(\infty),
    \text{ group manifold } \Spin(\infty),\\
{\rm 3.}\ &\bigl (\Sp(\infty)\times \Sp(\infty)\bigr )/\diag\, \Sp(\infty),
    \text{ group manifold } \Sp (\infty),\\
{\rm 4.}\ &\SU (p + \infty)/\mathrm{S}(\U(p)\times \U(\infty)),\
    \C^p \text{ subspaces of } \C^\infty, \\
{\rm 5.}\ &\SU (2\infty)/[\mathrm{S}(\U (\infty) \times \U (\infty))],\
    \C^\infty \text{ subspaces of infinite codim in } \C^\infty, \\
{\rm 6.}\ &\SU (\infty)/\SO (\infty),\
    \text{ real forms of } \C^\infty \\
{\rm 7.}\ &\SU (2\infty)/\Sp (\infty), \
    \text{ quaternion vector space structures on } \C^\infty,\\
{\rm 8.}\ &\SO (p + \infty)/[\SO (p)\times \SO (\infty)],\text{ oriented }
    \R^p \text{ subspaces of } \R^\infty, \\
{\rm 9.}\ &\SO (2\infty)/[\SO (\infty)\times \SO (\infty)], \
    \R^\infty \text{ subspaces of infinite codim in } \R^\infty, \\
{\rm 10.}\ &\SO (2\infty)/\U (\infty), \
    \text{ complex vector space structures on } \R^\infty, \\
{\rm 11.}\ &\Sp (p + \infty)/[\Sp(p)\times \Sp(\infty)],\
    \H^p \text{ subspaces of } \H^\infty, \\
{\rm 12.}\ &\Sp (2\infty)/[\Sp(\infty)\times \Sp(\infty)], \
    \H^\infty \text{ subspaces of infinite codim in } \H^\infty, \\
{\rm 13.}\ &\Sp (\infty)/\U (\infty), \
    \text{ complex forms of } \H^\infty.
\end{aligned}
\end{equation}

We also have as before injective systems
$\fg_n\hookrightarrow \fg_k$, $\fk_n\hookrightarrow \fk_k$, $\fs_n\hookrightarrow \fs_k$,
and $\fa_n\hookrightarrow \fa_k$ giving rise to corresponding injective systems. Let
$$
\mfg_\infty :=\varinjlim \mfg_n\, ,\quad \gk_\infty :=\varinjlim \gk_n\, ,\quad
\gs_\infty :=\varinjlim\gs_n\, , \quad \ga_\infty := \varinjlim\ga_n\, ,
 \quad\text{and}, \quad \fh_\infty :=\varinjlim \fh_n\, .$$
We have that $\mfg_\infty=\gk_\infty \oplus \gs_\infty$ is the eigenspace decomposition of
$\mfg_\infty $ with respect to the involution $\theta_\infty :=\varinjlim \theta_n$,
$\fa_\infty$ is a maximal abelian subspace of $\fs_\infty$.

We have also projective systems $\{\PW_{r}(i\fa_n)^{W_n}\}$, and $\{C_{r,W_n}(M_n)^{K_n}\}$
with surjective projections.
Let
\begin{eqnarray*}
\PW_r(\fa^{\infty *}_\C )&:=&\varprojlim \PW_{r}(\fa_{n,\C}^*)^{W_n}\\
C_{r,W_\infty}(M_\infty)^{K_\infty}
&:=&\varprojlim C_{r,W_n}(M_n)^{K_n} \, .
\end{eqnarray*}
As before we view the elements of
$\PW_r(\fa^*_{\infty,_\C})$ as
$W_\infty$--invariant functions on $\fa_{\infty,\C}^*$, and the elements of
of $C_{r,W_\infty}(M_\infty)^{K_\infty}$ as $K_\infty$--invariant functions
on $M_\infty$; see Remark \ref{inf-holo}. For $\mathbf{f}=(f_n)_n\in
C_{r,W_\infty}(M_\infty)^{K_\infty}$ define $\cS_{\rho,\infty} (\mathbf{f})
\in
\PW_r(\fa_{\infty,\C}^*)$ by
\begin{equation}
\cS_{\rho,\infty} (\mathbf{f}):=\{\cS_{\rho,n}(f_n)\}\, .
\end{equation}
Then $\cS_{\rho,\infty} (\mathbf{f})$ is well defined by Theorem \ref{th-PWcompactII}
and we have a commutative diagram
$$\xymatrix{\cdots & C^\infty_{r,W_n}(M_n)^{K_n}\ar[d]_{\cS_{\rho,n}}&
C^\infty_{r,W_{n+1}}(M_{n+1})^{K_{n+1}}\ar[d]_{\cS_{\rho,n+1}}\ar[l]_(0.55){C_{n+1,n}}&
\ar[l]_(0.2){C_{n+2,n+1}}\qquad \cdots &
C_{r,W_\infty}(M_\infty)^{K_\infty}\ar[d]_{\cS_{\rho,\infty}}
 \\
\cdots & \PW_r(\fa^*_{n,\C})&  \PW_r(\fa^*_{n+1\C}) \ar[l]^(0.55){P_{n+1,n}}
&\ar[l]^(0.2){P_{n+2,n+1}} \qquad \cdots & \PW_r(\fa_{\infty,\C}^* )\\
}\, ,
$$
see also \cite{OW,KW2009} for the spherical Fourier transform and
direct limits.

\begin{theorem}[Infinite dimensional Paley-Wiener Theorem-II]
\label{th-ProjLimCompact}
Let the notation be as above. Then
$\PW_r(\fa_{\infty,\C}^*)\not=\{0\}$,
$C_{r,W_\infty}(M_\infty)^{K_\infty}\not=\{0\}$,
and the spherical Fourier transform
\[\cF_\infty :C_{r,W_\infty}(M_\infty)^{K_\infty}
\to \PW_r(\fa^{\infty *}_\C )\]
is injective.
\end{theorem}

\section{Comparison with the $L^2$ Theory}\label{sec8}
\noindent
The maps considered up to this point are based on $C^\infty$  and
$C^\infty_c$ spaces rather than $L^2$ spaces and
unitary representation theory.  It is standard that $L^2$ for a
compact symmetric space is just a Hilbert space completion of the
corresponding $C^\infty$ space, and it turns out \cite[Proposition 3.27]{W2009}
that the same is true for inductive limits of compact symmetric spaces.
Here we discuss those inductive limits; any consideration of the
projective limit of $L^2$ spaces follows similar lines by replacing
the the maps of the inductive limit by the corresponding orthogonal
projections, because inductive and projective limits are the
same in the Hilbert space category.

The material of this section is taken from \cite[Section 3]{W2008a}
and \cite[Section 3]{W2009}
and adapted to our setting.  We always assume without further comments that
all extensions are propagations.

There are three steps to the comparison.  First, we describe the construction
of a direct limit Hilbert space
$L^2(M_\infty) := \varinjlim \{L^2(M_n),L_{m,n}\}$ that carries a natural
multiplicity--free unitary action of $G_\infty$.  Then we describe the
ring $\cA(M_\infty) := \varinjlim \{\cA(M_n),\nu_{m,n}\}$
of regular functions on $M_\infty$ where $\cA(M_n)$ consists of the
finite linear combinations of the matrix coefficients of the $\pi_\mu$
with $\mu \in \Lambda_n^+(G_n,K_n)$ and such that $\nu_{m,n}(f)|_{M_n} = f$.
Thus $\cA(M_\infty)$ is a (rather small) $G_\infty$--submodule of the projective
limit $\varprojlim\{\cA(M_n),\text{ restriction}\}$.  Third, we describe
a $\{G_n\}$--equivariant morphism
$\{\cA(M_n),\nu_{m,n}\} \rightsquigarrow \{L^2(M_n),L_{m,n}\}$ of
direct systems that embeds $\cA(M_\infty)$ as a dense $G$--submodule of
$L^2(M_\infty)$, so that $L^2(M_\infty)$ is $G_\infty$--isomorphic to a Hilbert
space completion of the function space $\cA(M_\infty)$.

We recall first some basic facts about the vector valued Fourier transform on
$M_n$ as well as the decomposition of $L^2(M_n)$ into irreducible summands.
To simplify notation write $\Lambda^+_n$ for $\Lambda^+(G_n,K_n)$.
Let $\mu \in\Lambda_n^+$ and let $V_{n,\mu}$ denote the irreducible
$G_n$--module of highest weight $\mu$. Recursively in $n$, we
choose a highest weight vector $v_{n,\mu} \in V_{n,\mu}$ and
and a $K_n$--invariant unit vector $e_{n,\mu} \in V_\mu^{K_n}$
such that
(i) $V_{n-1,\mu} \hookrightarrow V_{n,\mu}$ is isometric and
$G_{n-1}$--equivariant and sends $v_{n-1,\mu}$ to a multiple
of $v_{n,\mu}$, (ii) orthogonal projection $V_{n,\mu} \to V_{n-1,\mu}$
sends $e_{n,\mu}$ to a non--negative real
multiple $c_{n,n-1,\mu}e_{n-1,\mu}$ of $e_{n-1,\mu}$, and (iii)
$\langle v_{n,\mu},e_{n,\mu}\rangle =1$.
(Then $0 < c_{n,n-1,\mu} \leqq 1$.)  Note that orthogonal projection
$V_{m,\mu} \to V_{n,\mu}, m \geqq n$, sends
$e_{m,\mu}$ to $c_{m,n,\mu}e_{n,\mu}$ where
$c_{m,n,\mu} = c_{m,m-1,\mu}\cdots c_{n+1,n,\mu}$.

The Hermann Weyl degree formula provides polynomial functions on $\ga_\C^*$
that map $\mu$ to $\deg(\pi_{n,\mu}) = \dim V_{n,\mu}$.
Earlier in this paper we had written
$\deg(\mu)$ for that degree when $n$ was fixed, but here it is crucial to
track the variation of $\deg(\pi_{n,\mu})$ as $n$ increases.
Define a map $v\mapsto f_{n,\mu, v}$ from $V_{n,\mu}$ into
$L^2(M_n)$ by
\begin{equation}\label{def-Coeff}
f_{n,\mu,v}(x) = \langle v,\pi_{n,\mu}(x)e_\mu \rangle \,.
\end{equation}
It follows by the Frobenius--Schur orthogonality relations that
$v\mapsto \deg(\pi_{n,\mu} )^{1/2} f_{\mu,v}$ is a unitary
$G_n$ map from $V_\mu $ onto its image in $L^2(M_n)$.

The operator valued Fourier transform
$$L^2(G_n) \to \bigoplus_{\mu\in \Lambda^+_n}\mathrm{Hom} (V_{n,\mu},V_{n,\mu})
\cong \bigoplus_{\mu\in \Lambda^+_n}V_{n,\mu}\otimes V_{n,\mu}^*$$
is defined by $f \mapsto \bigoplus_{\mu\in \Lambda^+_n}\pi_{n,\mu} (f)$ where
$\pi_{n,\mu} (f) \in \mathrm{Hom} (V_{n,\mu},V_{n,\mu})$ is given by
\begin{equation}
\pi_{n,\mu} (f) v := \int_{G_n} f(x) \pi_{n,\mu} (x)v\,\text{ for } f\in L^2(G_n)\,  .
\end{equation}
Denote by $P^{K_n}_\mu $ the orthogonal projection
$V_{n,\mu} \to V_{n,\mu} ^{K_n}$. Then
$P^{K_n}_\mu (v)=\int_{K_n}\pi_{n,\mu} (k)v\, dk$,
and if $f$ is right $K_n$--invariant, then
$$\pi_{n,\mu} (f)=\pi_{n,\mu} (f)P^{K_n}_\mu\, .$$
That gives us the vector valued Fourier transform
$f\mapsto \widehat{f}: \Lambda^+_n\to \bigoplus_{\mu\in\Lambda_n^+} V_{n,\mu}$\,,
\begin{equation}
L^2(M_n)\to \bigoplus_{\mu\in\Lambda_n^+} V_{n,\mu}
\text{ defined by } f\mapsto \widehat{f}(\mu ) :=\pi_{n,\mu} (f)e_{n,\mu}\, .
\end{equation}
Then the Plancherel formula for $L^2(M_n)$ states that
\begin{equation}\label{e-FourierExpansion}
f=\sum_{\mu\in \Lambda_n^+} \deg(\pi_{n,\mu} ) f_{\mu ,\widehat{f}(\mu )}
=\sum_{\mu\in \Lambda_n^+} \deg(\pi_{n,\mu} ) \langle
\widehat{f}(\mu ), \pi_{n,\mu} (\, \cdot \, )e_{n,\mu} \rangle
\end{equation}
in $L^2(M_n)$ and
\begin{equation}\label{e-Norm}
\|f\|^2_{L^2}=\sum_{\mu\in \Lambda^+_n} \deg(\pi_{n,\mu} )
\|\widehat{f}(\mu )\|_{HS}^2\, .
\end{equation}
If $f$ is smooth, then the series in (\ref{e-FourierExpansion}) converges
in the $C^\infty$ topology of $C^\infty (M_n)$.

For $n\leqq m$ and $\mu = \mu_{I,n}\in \Lambda^+_n$ consider the following
diagram of unitary $G_n$-maps, adapted from \cite[Equation 3.21]{W2009}:
$$
\xymatrix{V_{\mu_{I,n}}
\ar[d]_{v\mapsto \deg(\pi_{n,\mu})^{1/2} f_{\mu_{I,n},v}}\ar[r]^{v\mapsto v}
& V_{\mu_{I,m}}\ar[d]^{v\mapsto \deg(\pi_{m,\mu})^{1/2}  f_{\mu_{I,m},v}}\\
L^2(M_n) \ar[r]_{L_{m,n}}& L^2(M_m)   }
$$
where $L_{m,n}: L^2(M_n)\to L^2(M_k)$ is the $G_n$--equivariant
partial isometry defined by
\begin{equation}\label{def-psi(m,n)}
L_{k,n}: \sum_{I_n} f_{\mu_{I,n},w_{I}}\mapsto
 \sum_{I_m} c_{m,n,\mu}\sqrt{\tfrac{\deg (\pi_{m,\mu})}{\deg (\pi_{n,\mu})}}\,
f_{\mu_{I,m},w_{I}}\, ,\quad w_I\in V_{n,\mu}\, .
\end{equation}
As in \cite[Section 4]{W2009} we have
\begin{theorem}\label{fun-restriction}
The map $L_{k,n}$ of {\rm (\ref{def-psi(m,n)})}
is a $G_n$--equivariant partial isometry with image
$$
\Im (L_{m,n} )\cong {\bigoplus}_{I\in (\Z^+)^{r_k},\
k_{r_n+1}=\ldots = k_{r_k}=0} V_{\mu_{I}}\, .
$$
If
$n\leqq m \leqq k$ then
$$L_{k,n}=L_{m,n}\circ L_{k,m}$$
making $\{L^2(M_n),L_{k ,n}\}$ into
a direct system of Hilbert spaces.
\end{theorem}

\noindent
Define
\begin{equation}
L^2(M_\infty ) := \varinjlim L^2(M_n),
\end{equation}
direct limit in the category of Hilbert spaces and unitary injections.

From construction of the $L_{m,n}$ we now have
\begin{theorem}[\cite{W2008a}, Theorem 13] \label{cor-symm-mfree}
The left regular representation of $G_\infty$ on $L^2(M_\infty)$
is a multiplicity free discrete direct sum of
irreducible representations.  Specifically, that left regular
representation is $\sum_{I \in \cI} \pi_I$ where
$\pi_I = \varinjlim \pi_{I,n}$ is the irreducible representation
of $G_\infty$ with highest weight $\xi_I := \sum k_r\xi_r$.  This
applies to all the direct systems of {\rm (\ref{e-infiniteDim})}.
\end{theorem}

The problem with the partial isometries $L_{m,n}$ is that they do not
work well with restriction of functions, because of the rescalings
and because $L_{m,n}(L^2(M_n)^{K_n}) \not\subset L^2(M_m)^{K_m}$
for $n < m$.
In particular the spherical functions $\psi_{I,n}(g) :=
\langle e_{I,n} ,\pi_{I,n}(g)e_{I,n})\rangle$ do not map
forward, in other words $L_{m,n}(\psi_{I,n})\not=\psi_{I,m}$.

We deal with this by viewing $L^2(M_\infty)$ as a Hilbert space completion
of the ring $\cA(M_\infty) := \varinjlim \cA(M_n)$ of regular functions on
$M_\infty$.  Adapting \cite[Section 3]{W2009} to our notation, we define
\begin{equation}\label{symm-quo-lim-coef}
\begin{aligned}
&\cA(\pi_{n,\mu})^{K_n} = \{\text{finite linear combinations of the }
  f_{\mu,I_n,w_I}
  \text{ where } w_I \in V_{n,\mu} \}, \\
&\nu_{m,n,\mu}:\cA(\pi_{n,\mu})^{K_n} \hookrightarrow
  \cA(\pi_{m,\mu})^{K_m} \text{ by }
        f_{\mu,I_n,w_I} \mapsto f_{\mu,I_m,w_I}\,\, .
\end{aligned}
\end{equation}
Thus \cite[Lemma 2.30]{W2009} says that if
$f \in \cA(\pi_{n,\mu})^{K_n}$ then $\nu_{m,n,\mu}(f)|_{M_n} = f$.

The ring of regular functions on $M_n$ is
$\cA(M_n) := \cA(G_n)^{K_n}
= \sum_\mu \cA(\pi_{n,\mu})$, and the $\nu_{m,n,\mu}$ sum to
define a direct system $\{\cA(M_n),\nu_{m,n}\}$.  Its limit is
\begin{equation}\label{symm-quo-lim-reg}
\cA(M_\infty) := \cA(G_\infty)^{K_\infty} =
\varinjlim \{\cA(M_n),\nu_{m,n}\}.
\end{equation}
As just noted, the maps of the direct system $\{\cA(M_n),\nu_{m,n}\}$
are inverse to restriction of functions, so $\cA(M_\infty)$ is a
$G_\infty$--submodule
of the inverse limit $\varprojlim \{\cA(M_n), \text{ restriction}\}$.

For each $n$,
$\cA(M_n)$ is a dense subspace of $L^2(M_n)$ but, because the
$\nu_{m,n}$ distort the Hilbert space structure,
$\cA(M_\infty)$ does not sit naturally as a subspace of $L^2(M_\infty)$.
Thus we use the $G_n$--equivariant maps
\begin{equation}\label{sym-rel-map-sys}
\eta_{n,\mu}: \cA(\pi_{n,\mu})^{K_n} \to
\cH_{\pi_n}\widehat{\otimes}(w_{n,\mu^*}\C) \text{ by }
f_{\mu,I_n,w_I} \mapsto c_{n,1,\mu}
\sqrt{\deg \pi_{n,\mu}}\, f_{\mu,I_n,w_I}.
\end{equation}
where $c_{m,n,\mu}$ is the length of the projection of $e_{m,\mu}$
to $V_{n,\mu}$.  Now \cite[Proposition 3.27]{W2009} says
\begin{proposition}\label{sym-quo-comparison}
The maps $L_{m,n,\mu}$ of {\rm (\ref{def-psi(m,n)})},
$\nu_{m,n,\mu}$ of {\rm (\ref{symm-quo-lim-coef})} and
$\eta_{n,\mu}$ of {\rm (\ref{sym-rel-map-sys})}
satisfy $$(\eta_{m,\mu}\circ \nu_{m,n,\mu})(f_{\mu,I_n,w_I}) =
(L_{m,n,\mu}\circ\eta_{n,\mu})(f_{\mu,I_n,w_I})$$
for $f_{u,v,n} \in \cA(\pi_{n,\mu})^{K_n}$.  Thus they inject
the direct system $\{\cA(M_n), \nu_{m,n}\}$
into the direct system $\{L^2(M_n),L_{m,n}\}$.
That map of direct systems defines a $G_\infty$--equivariant injection
$$
\widetilde{\eta}: \cA(M_\infty) \to L^2(M_\infty)
$$
with dense image.  In particular $\eta$ defines a pre Hilbert space
structure on $\cA(M_\infty)$ with completion isometric to $L^2(M_\infty)$.
\end{proposition}
This describes $L^2(M_\infty)$ as an ordinary Hilbert space completion of
a natural function space on $M_\infty$.

\end{document}